\documentclass[11pt]{amsart}
\usepackage[all]{xy}
\usepackage{amsmath,amsfonts,amssymb,amsthm,epsfig,amscd,psfrag,framed,comment,color}

\allowdisplaybreaks
\usepackage{xargs}
\usepackage{environ}
\usepackage{graphicx}

\usepackage{tikz}
\usetikzlibrary{shapes}
\usetikzlibrary{backgrounds}
\usetikzlibrary{decorations,decorations.pathreplacing,decorations.markings}
\usetikzlibrary{fit,calc,through}
\usetikzlibrary{external}

\newtheorem{Theorem}{Theorem}[section]
\newtheorem{Proposition}[Theorem]{Proposition}
\newtheorem{Lemma}[Theorem]{Lemma}

\newtheorem{Corollary}[Theorem]{Corollary}
\newtheorem{Conjecture}[Theorem]{Conjecture}

\theoremstyle{definition}
\newtheorem{Remark}[equation]{Remark}
\newtheorem{Example}[equation]{Example}

\tikzstyle{mid>}=[decoration={markings, mark=at position 0.5 with {\arrow{>}}}, postaction={decorate}]
\tikzstyle{mid<}=[decoration={markings, mark=at position 0.5 with {\arrow{<}}}, postaction={decorate}]
\tikzstyle{upper>}=[decoration={markings, mark=at position 0.8 with {\arrow{>}}}, postaction={decorate}]
\tikzstyle{upper<}=[decoration={markings, mark=at position 0.8 with {\arrow{<}}}, postaction={decorate}]
\tikzstyle{lower>}=[decoration={markings, mark=at position 0.2 with {\arrow{>}}}, postaction={decorate}]
\tikzstyle{lower<}=[decoration={markings, mark=at position 0.2 with {\arrow{<}}}, postaction={decorate}]

\newcommandx{\NewEnvironx}[5][2,3]{%
  \expandafter\newcommandx\csname start#1\endcsname[#2][#3]{#4}%
  \NewEnviron{#1}{\csname start#1\expandafter\endcsname\BODY #5}}

\newcommand{\ladderX}{1.5}
\newcommand{\ladderY}{1.5}
\newcommand{\ladderR}{0.6}
\newcommand{\laddercoordinates}[2]{
\foreach \x in {0,...,#1} {
        \foreach \y in {0,...,#2} {
                \coordinate (l\x\y) at (\x * \ladderX, \y * \ladderY);
                \coordinate (u\x\y) at ($(l\x\y)+\ladderR*(0,\ladderY)$);
                \coordinate (d\x\y) at ($(l\x\y)+(0,\ladderY)-\ladderR*(0,\ladderY)$);
        }
}
}

\NewEnvironx{ladder}[2]{%
  \begin{tikzpicture}[baseline=13*\ladderY*#2]\laddercoordinates{#1}{#2}}
{\end{tikzpicture}}

\newcommand{\scs}{\scriptstyle}

\newcommand{\qbins}[2]{
\left[ \;
\vcenter{\xy
(0,3)*{\scs #1}; (0,-1)*{\scs #2};
\endxy} \;
 \right]
}

\def\sA{{\sf{A}}}
\def\id{{\rm{id}}}
\def\diag{{\rm{diag}}}
\def\sI{\mathcal{I}}
\def\ABr{ABr}

\def\l{\lambda}
\def\O{\mathcal{O}}
\def\ch{{\rm{ch}}}

\def\GrBD{{\sf{Gr}}}

\def\Spec{{\rm Spec}}

\def\t{{\theta}}
\def\oD{\overline{D}}

\def\sE{{\mathcal{E}}}
\def\sF{{\mathcal{F}}}
\def\E{{\sf{E}}}
\def\F{{\sf{F}}}
\def\T{{\sf{T}}}
\def\Cone{{\rm{Cone}}}
\def\ux{{\underline{x}}}
\def\uy{{\underline{y}}}
\def\sT{{\mathcal{T}}}

\def\G{{\mathbb{G}}}

\def\Sym{{\rm Sym}}
\def\sP{{\mathcal{P}}}
\def\sQ{{\mathcal{Q}}}
\def\sF{{\mathcal{F}}}
\def\sL{{\mathcal{L}}}
\def\sV{{\mathcal{V}}}
\def\sM{{\mathcal{M}}}
\def\la{\langle}
\def\ra{\rangle}
\def\l{\lambda}
\def\A{{\mathbb{A}}}
\def\P{{\mathbb{P}}}

\def\spn{\operatorname{span}}

\newcommand{\g}{\mathfrak{g}}

\renewcommand{\sl}{\mathfrak{sl}}
\newcommand{\gl}{\mathfrak{gl}}
\newcommand{\C}{\mathbb{C}}

\newcommand{\Cx}{{\mathbb{C}^\times}}
\newcommand{\K}{\mathcal{K}}
\newcommand{\tK}{\widetilde{\mathcal{K}}}
\newcommand{\N}{\mathbb{N}}

\newcommand{\codim}{\operatorname{codim}}
\newcommand{\disj}{\mathrm{disj}}

\def\Kom{{\sf{Kom}}}
\def\hsl{\widehat{\mathfrak{sl}}}
\def\hgl{\widehat{\mathfrak{gl}}}
\def\1{{{\bf 1}}}

\def\cX{{\mathcal{X}}}
\def\cY{{\mathcal{Y}}}

\DeclareMathOperator{\Hom}{Hom}
\DeclareMathOperator{\End}{End}

\def\D{\mathbb D}
\def\bZ{\mathbb{{Z}}}
\def\Z{{\mathbb{Z}}}
\def\W{{\mathbb{W}}}
\def\Y{{\mathbb{Y}}}
\def\X{{\mathbb{X}}}

\def\uk{{\underline{{k}}}}

\def\ul{{\underline{l}}}
\def\mod{{\text{-}\mathrm{mod}}}

\begin{document}
\setcounter{tocdepth}{1}

\title{Categorical geometric symmetric Howe duality}

\author{Sabin Cautis}
\email{cautis@math.ubc.ca}
\address{Department of Mathematics\\ University of British Columbia \\ Vancouver BC, Canada}

\author{Joel Kamnitzer}
\email{jkamnitz@math.utoronto.ca}
\address{Department of Mathematics \\ University of Toronto \\ Toronto ON, Canada}

\begin{abstract}
We provide a natural geometric setting for symmetric Howe duality. This is realized as a (loop) $\sl_n$ action on derived categories of coherent sheaves on certain varieties arising in the geometry of the Beilinson-Drinfeld Grassmannian.

The main construction parallels our earlier work on categorical $\sl_n$ actions and skew Howe duality. In that case the varieties involved arose in the geometry of the affine Grassmannian. We discuss some relationships between the two actions.
\end{abstract}

% \date{\today}
\maketitle
\tableofcontents

\section{Introduction}

\subsection{Skew Howe duality}
Consider the vector space $ \Lambda^N(\C^n \otimes \C^m) $, which comes with action of the groups $ SL_n $ and $ SL_m $.  These two actions commute and their action generate each other's commutant.  This is known as skew Howe duality.

In recent years, there has been an increasing interest in skew Howe duality from the viewpoint of categorification and knot invariants.

A geometric framework for skew Howe duality was developed by Mirkovi\'c-Vybornov \cite{MVy}. In their paper, they explained that the vector space $ \Lambda^N(\C^n \otimes \C^m) $ arises (via the geometric Satake correspondence) as the direct sum of the homology of certain convolutions  $Y(\uk) $ constructed using the affine Grassmannian of $ PGL_m $.
$$
\Lambda^N(\C^n \otimes \C^m) = \bigoplus_{\substack{\uk =(k_1, \dots, k_n) \\ k_1 + \dots + k_n = N}} \Lambda^{k_1} \C^m \otimes \cdots \otimes \Lambda^{k_n} \C^m = \bigoplus_{\uk} H_*(Y(\uk))
$$
From this perspective, $ \sl_n $ acts on $  \bigoplus_{\uk} H_*(Y(\uk)) $ using natural correspondences.  In \cite{CK4}, we showed precisely how these correspondences express the $ \sl_n $ action. We also considered a quantum version using equivariant K-theory.  In particular, we constructed an equivalence of categories between $ KConv^{SL_m \times \C^\times}({\GrBD}) $ and $\O_q^{min}(\frac{SL_m}{SL_m})\mod$ where $ KConv $ denotes the category whose morphisms are given by equivariant $K$-theory of fibre products $ Y(\uk) \times_{\mathsf{Gr}} Y(\uk') $.

In our paper \cite{CKL1}, we gave a categorification of Mirkovi\'c-Vybornov's geometric framework.  More precisely, we constructed a categorical $\sl_n $ action on the derived categories of coherent sheaves $ \oplus_{\uk} D(Y(\uk)) $ of these varieties.  (Actually \cite{CKL1} only dealt with the case $n = 2$, which is the most difficult case, while papers \cite{CK3}, \cite{C1} extended this to arbitrary $n$.)

The motivation in \cite{CKL1} was to develop the skew Howe duality story in order to extend the construction of homological knot invariants using the affine Grassmannian that we began in \cite{CK1}. We completed this in \cite{C1} where we categorified the Restikhin-Turaev invariants of $\sl_m$.  More recently, the idea of using categorical skew Howe duality to produce (and compare) homological knot invariants has been influential (see \cite{MW,QR} for example).

\subsection{Symmetric Howe duality}
There is a counterpart to skew Howe duality, where we replace $ \Lambda^N(\C^n \otimes \C^m) $ with $  \Sym^N(\C^n \otimes \C^m) $.  From the perspective of the geometric Satake correspondence, this means replacing convolutions of smooth Schubert varieties $ \GrBD^{\omega_k} $ with convolutions of singular Schubert varieties $ \overline{\GrBD^{k\omega_1}} $. These singularities make working with derived categories of coherent sheaves more difficult (for example, derived tensoring becomes very complicated and we need to pass to unbounded categories). However, these singular Schubert varieties (and their convolutions) have natural deformations, which we denote $\Y(\uk)$, defined using the Beilinson-Drinfeld Grassmannian.

Like the $Y(\uk)$, these varieties $\Y(\uk)$ have natural linear algebra descriptions, given in Section \ref{sec:varieties}. However, their geometry is a little more complicated to study. For example, it turns out that $\Y(\uk)$ are smooth (which is important for our purposes) but it takes us a good part of Section \ref{sec:geometry} to prove this.

The main idea in this paper is to use derived categories of coherent sheaves on these varieties to categorify $ \Sym^N(\C^n \otimes \C^m)$.  Our main result (Theorem \ref{thm:main}) can be loosely stated as follows.

\begin{Theorem}
There exists a categorical $ L\sl_n $ action on
$$ \bigoplus_{\substack{\uk = (k_1, \dots, k_n) \\ k_1 + \dots + k_n = N}} D(\Y(\uk)) $$
which categorifies the representation $ \Sym^N(\C^n \otimes \C^m) $ of the quantum loop algebra $ U_q(L\sl_n)$.
\end{Theorem}

Notice that the action above is of the loop algebra $U_q(L\sl_n)$. Strictly speaking only the finite part $U_q(\sl_n)$ plays a role in Howe duality. However, as is (almost always) the case with geometric actions of Lie algebras (for instance \cite{N} in the case of quiver varieties or \cite{CK4} in the case of $Y(\uk)$) the action naturally extends to a loop algebra action by using line bundles. For completeness we define and study this larger action in this paper. In order to do this we extend the notion of a categorical $(\sl_n,\t)$ action to an $(L\sl_n,\t)$ action in Section \ref{sec:catactions}. Roughly speaking, this means having an action of $\E_i,\F_i,\E_{i,1}$ and $\F_{i,-1}$ where $i \in \{1,\dots,n-1\}$ and which satisfy some relations.

Despite the similarities with categorical geometric skew Howe duality, our symmetric setting is of a different flavour. In particular, the shift functor $\la 1 \ra$ that appears in the definition of the categorical action is mapped to a purely equivariant shift $\{1\}$ (instead of a homological-equivariant shift $[1]\{-1\}$ on the skew side). One reason this is interesting, and one of the motivations for writing this paper, is that it will allow us to define exotic t-structures in the sense of \cite{BM}, as we now briefly explain.

The $(L\sl_n,\t)$ action in this paper is manifestly in its loop (or Drinfeld-Jimbo) realization. One can pass to the Kac-Moody presentation to obtain a categorical $(\hsl_n,\t)$ action. This means that we have generators $\E_i,\F_i$ for $i \in \{0,1,\dots,n-1\}$.  In \cite{CaKo} we will use this action to show that there exists an essentially unique ``exotic'' t-structure on $\oplus_\uk D(\Y(\uk))$ such that these generators $\E_i,\F_i$ for $i \in \{0,1,\dots,n-1\}$ are t-exact.

\subsection{Analogy with the Springer resolution}
This paper was motivated in part by the analogy with the Springer resolution as we now explain.

The $(\sl_n,\t)$ action in our main theorem leads (using \cite{CK3}) to an action of the finite braid group on $\oplus_{\uk} D(\Y(\uk))$. In Section \ref{sec:braids} we explain how the larger $(L\sl_n,\t)$ can be used to extend this to an affine braid group action on $\oplus_{\uk} D(\Y(\uk))$ (this extends the main result of \cite{CK3}).

On the other hand, for any semisimple Lie algebra, we can consider the Springer resolution $ \widetilde{\mathcal N} $ and its Grothendieck extension $ \widetilde{\g}$.  Recall that Bezrukavnikov-Riche \cite{BR} constructed actions of the affine braid group on $ D(\widetilde{\mathcal N}) $ and $ D(\widetilde{\g}) $. In their action, the generators of the finite braid group act by irreducible correspondences on $ D(\widetilde{\g}) $ (but not on $ D(\widetilde{\mathcal N}) $).

Our varieties $ \Y(\uk) $ are analogous to $ \widetilde{\g} $, just as our skew Howe duality varieties $ Y(\uk) $ were analogous to $ \widetilde{\mathcal N} $. For example, the role of the morphism $ \widetilde{\g} \rightarrow \mathfrak{h} $ is played by a morphism $ \ch : \Y(\uk) \rightarrow \A_{\uk} $ (where $ \A_\uk $ is a partial symmetrization of $\A^N$). When $ n = m $ and $ \uk = 1^n $ the variety $ \Y(\uk) $ actually contains an open subset isomorphic to $ \widetilde{\g} $ where $\g = \gl_n$ and this analogy can be made precise.

In particular, in Section \ref{sec:equiv}, we give an explicit description of the generators of the affine braid group acting on $\oplus_\uk D(\Y(\uk))$. It then follows that this action agrees with that from \cite{BR} if we restrict to the open subset $\widetilde{\g} \subset \Y(1^n)$. This also implies that the exotic t-structure on $D(\Y(1^n))$ that we briefly discussed above restricts to the exotic t-structure on $D(\widetilde{\g})$ that was studied in \cite{BM}.

\subsection{Relation to Skew-Sym Howe duality}
In \cite{CKM}, we gave a presentation of the subcategory of $ SL_m $ representations whose objects are tensor products of exterior powers of $ \C^m $.  It is pretty easy to construct all the generating morphisms in this category; the hard part is to determine the relations.  For this step, we used the fact that via skew Howe duality, $ \End_{SL_m} (\Lambda^N(\C^n \otimes \C^m)) $ is a generalized Schur quotient of $ U(\sl_n) $.

It is natural to try to use symmetric Howe duality to give a presentation of the subcategory of $ SL_m $ representations whose objects are tensor products of symmetric powers of $ \C^m $.  However, the situation is more complicated here, because although $  \End_{SL_m} (\Sym^N(\C^n \otimes \C^m)) $ is a quotient of $ U (\sl_n) $, it is not a Schur quotient, so it is not clear what the relations should be.

This problem was recently solved by Rose and Tubbenhauer \cite{RT} in the $\sl_2$ case and then extended by Tubbenhauer, Vaz and Wedrich \cite{TVW} to $\sl_n$. Their idea was to consider symmetric and exterior powers together and then use a single ``dumbbell'' relation (which expresses that $ \C^m $ is both a symmetric and an exterior power) to ``glue'' the two sides together. This gives a presentation of the subcategory whose objects are tensor products of both symmetric and exterior powers.

In this paper, we find a geometric incarnation of their theory by considering the varieties $\Y(\uk)$ and $Y(\uk)$. In Section \ref{sec:skewsym}, we prove that the inclusion morphism $\iota: Y(1^n) \rightarrow \Y(1^n) $ intertwines the braid group actions on both sides. We then explain how this leads to a geometrization (or categorification) of the dumbbell relation.

\subsection{Acknowledgements}
We would like to thank Roman Bezrukavnikov, Alexander Braverman, Vinoth Nandakumar and David Rose for helpful conversations. S.C. was supported by an NSERC discovery/accelerator grant and J.K. was supported by an NSERC discovery/accelerator grant and a Sloan fellowship.

\section{Geometric setup}\label{sec:setup}

In this section we review our conventions, define the varieties $\Y(\uk)$ and the correspondences between them that will be used to construct the categorical $\sl_n$ action.

\subsection{Notation}

For a smooth variety $X$ equipped with a $\C^\times$-action, let $D(X)$ denote the bounded derived category of $\C^\times$-equivariant coherent sheaves. We denote by $\O_X \{k\}$ the structure sheaf of $X$ with non-trivial $\C^\times$ action of weight $k$. More precisely, if $f \in \O_X(U)$ is a local function then, viewed as a section $f' \in \O_X \{k\}(U)$, we have $t \cdot f' = t^{-k}(t \cdot f)$. More generally, $\sM \{k\} := \sM \otimes \O_X \{i\}$.

Suppose $Y$ is another smooth variety with a $\C^\times$ action with an object $\sP \in D(X \times Y)$ whose support is proper over $Y$ (we will always assume this is the case from hereon). Then $\sP$ induces a functor $\Phi_{\sP}: D(X) \rightarrow D(Y)$ via $(\cdot) \mapsto \pi_{2*}(\pi_1^* (\cdot) \otimes \sP)$. We say $\sP$ is the kernel which induces $\Phi_{\sP}$.

Its right and left adjoints $\Phi_{\sP}^R$ and $\Phi_{\sP}^L$ are induced by $\sP^R := \sP^\vee \otimes \pi_2^* \omega_X [\dim(X)]$ and $\sP^L := \sP^\vee \otimes \pi_1^* \omega_Y [\dim(Y)]$ respectively. Moreover, if $\sQ \in D(Y \times Z)$ then $\Phi_{\sQ} \circ \Phi_{\sP} \cong \Phi_{\sQ * \sP}: D(X) \rightarrow D(Y)$ where $\sQ * \sP = \pi_{13*}(\pi_{12}^* \sP \otimes \pi_{23}^* \sQ)$ is the convolution product.

\subsection{The varieties} \label{sec:varieties}

Fix positive integer $m$. For a sequence of natural numbers $\uk = (k_1, \dots, k_n)$ define the varieties $\Y(\uk)$ by
\begin{equation*}
\Y(\uk) := \{ \C[z]^m = L_0 \subseteq L_1 \subseteq \dots \subseteq L_n \subseteq \C(z)^m : z L_i \subset L_i, \dim(L_i/L_{i-1}) = k_i \}
\end{equation*}
where the $L_i$ are complex vector subspaces. For convenience we sometimes encode this variety as
$$\{L_0 \xrightarrow{k_1} L_1 \xrightarrow{k_2} \dots \xrightarrow{k_n} L_n\}$$
where the superscripts denote the codimension of the inclusion and the condition $zL_i \subset L_i$ is implicit.

On $\Y(\uk)$ we have natural vector bundles $\sV_i$ whose fiber over a point $(L_0 \subseteq \dots \subseteq L_n)$ is the vector space $L_i/L_0$. By forgetting $L_n$ we also have projections
$$\Y(k_1, \dots, k_{n-1}, k_n) \rightarrow \Y(k_1, \dots, k_{n-1})$$
whose fibers are isomorphic to $\Y(k_n)$. Thus $\Y(\uk)$ is an iterated $\Y(k)$-bundle (for various $k \in \N$).

Now consider $\Y(k) = \{L_0 \subset L_1: zL_1 \subset L_1, \dim(L_1/L_0)=k\}$. Then the characteristic polynomial of $z|_{L_1/L_0}$ gives us a (proper) map $\ch: \Y(k) \rightarrow \A_k$ where $\A_k := \A^k/S_k$ is the quotient by the symmetric group. More generally, this gives us a map $\ch: Y(\uk) \rightarrow \A_{\uk}$ where $\A_{\uk} := \A_{k_1} \times \dots \times \A_{k_n}$.

There is an action of $\C^\times$ on $\C(z)$ given by $t \cdot z^k = t^{2k} z^k$. This induces an action of $\C^\times$ on $\C(z)^m$. It is not hard to check that this also induces a $\C^\times$ action on $\Y(\uk)$. The map $\ch: \Y(\uk) \rightarrow \A_\uk$ is then $\C^\times$-equivariant if we equip $\A^1 = \Spec \C[x]$ with the dilation action $t \cdot x = t^2 x$ and $\A_k = (\A^1)^k/S_k$ with the corresponding action. We will always work $\C^\times$-equivariantly with respect to this action.

Now consider $\Y(k) \rightarrow \A_k = \Spec \C[e_1, \dots, e_k]$ where the $e_i$ are the elementary symmetric functions in the eigenvalues of $z$. Then multiplication by $e_i$ defines a map of degree $2i$, $\O_{\Y(k)} \rightarrow \O_{\Y(k)} \{2i\}$ (there is also the obvious generalization to $\Y(\uk)$). Moreover, the action of $z$ defines a degree $2$ endomorphism $z: \sV_i \rightarrow \sV_i \{2\}$ on the vector bundles $\sV_i$ on $\Y(\uk)$.

\subsection{$\Y(\uk)$ vs $Y(\uk)$}\label{sec:Ys}
In \cite{CK2} we defined the similar varieties
$$
Y(\uk) := \{ \C[[z]]^m = L_0 \subseteq L_1 \subseteq \dots \subseteq L_n \subseteq \C((z))^m : z L_i \subset L_{i-1}, \dim(L_i/L_{i-1}) = k_i \}
$$
(there we used the notation $Y_\beta$ rather than $Y(\uk)$). These varieties are non-empty only if $ 0 \le k_i \le m $ for all $i$, in which case they are iterated Grassmannian bundles.  (On the other hand, the $\Y(\uk)$ are non-empty for any $k_i \in \N$.)

We have an inclusion $ Y(\uk) \subset \Y(\uk) $ as follows.  First note that if $ \C[z]^m \subset L $ is a $ \C[z] $ submodule of $ \C(z)^m $, then we can form the base change $$ R^m \subset L_R = L \otimes_{\C[z]} R \subset \C(z)^m \otimes_{\C[z]} R $$ for any ring $ R $ containing $ \C[z] $.  In particular, if $ R = \C[[z]] $, then the map $ L \mapsto L_R $ gives us an identification
$$ \ch^{-1}(0) = \{ \C[[z]]^m = L_0 \subseteq L_1 \subseteq \dots \subseteq L_n \subseteq \C((z))^m : z L_i \subset L_i, \dim(L_i/L_{i-1}) = k_i \} $$
and thus we see that $Y(\uk) \subset \ch^{-1}(0) $.

\subsection{Correspondences} \label{sec:corr}

For $i \in I = \{1, \dots, n-1\}$ we write $\alpha_i$ for the sequence $(0,\dots,0,-1,1,0,\dots,0)$ where the $-1$ occurs in position $i$. We define correspondences $\Y_i^r(\uk) \subset \Y(\uk) \times \Y(\uk+r\alpha_i)$ by
\begin{equation*}
\Y_i^r(\uk) := \{ (L_\bullet, L'_\bullet) : L_j = L_j' \text{ for } j \ne i, L_i' \subset L_i \}.
\end{equation*}
In other words, $\Y_i^r(\uk)$ is the variety
$$\{L_0 \xrightarrow{k_1} L_1 \xrightarrow{k_2} \dots \xrightarrow{k_{i-1}} L_{i-1} \xrightarrow{k_i-r} L_i' \xrightarrow{r} L_i \xrightarrow{k_{i+1}} L_{i+1} \xrightarrow{k_{i+2}} \dots \xrightarrow{k_n} L_n \}$$
As before we also have natural bundles $\sV_j=\sV_j'$ for $j \ne i$ as well as $\sV_i,\sV_i'$. Note that $\Y_i^r(\uk) \cong \Y(\ul)$ where $\ul = (k_1, \dots, k_i-r, r, k_{i+1}, \dots, k_n)$.

\section{Geometry of the $\Y$ spaces}\label{sec:geometry}

We would like to discuss the geometry of $\Y(\uk)$ in a little more detail. The main results in this section are Proposition \ref{prop:smooth} which implies that all the varieties $\Y(\uk)$ are smooth and Corollary \ref{cor:small} which shows that the natural projection maps $ \Y(\uk) \rightarrow \Y(k_1 + \dots + k_n) $ are small.

\subsection{Another version of $\Y(k)$}
In this section, we will work with an alternate version of $ \Y(k) $.  We define
$$ \X(k) = \{\C[z]^m = L_0 \supset L : zL \subset L, \dim(L_0/L)=k\}$$

Choose a non-degenerate symmetric bilinear form $(\cdot,\cdot) $ on $ \C^m$, which we extend $ \C(z)$-linearly to $ \C(z)^m$.  Then we given $ L \in \X(k) $, we define $$ L^\perp := \{ v \in \C(z)^m : (v, w) \in \C[z] \text{ for all $ w \in L $} \} $$
It is easy to see that $ L^\perp \in \Y(k) $ and that the map $ \X(k) \rightarrow \Y(k) $ is an isomorphism.  In this section, for convenience, we will work with $ \X(k) $.

It will be convenient to introduce a notation for the (disjoint) union of these spaces
\begin{equation*}
\X := \{\C[z]^m = L_0  \supset L : z L \subset L\, \dim L_0/L < \infty\} = \cup_k \X(k)
\end{equation*}

\subsection{An open subset}
For each $ p \in \N$, let
$$ W_p := \spn(e_1, \dots, z^{p-1}e_1, \dots, e_m, \dots, z^{p-1} e_m) \subset L_0 .$$
Fix $ k $ and let us write $ k = q m + r $ where $ 0 \le r < m $.

Given $ L \in \X(k) $, the codimension of $ L $ in $ L_0 $ is $ qm+r $, thus $ \dim L \cap W_{q+1} \ge m-r $, while $ \dim L \cap W_q \ge 0 $.  Thus, it is natural to consider
$$ \X(k)_0 = \{ L : L \cap W_q = 0,\ \dim L \cap W_{q+1} = m-r  \} $$
This is an open subset of $ \X(k)$.

Given a point $ L \in \X(k)_0$, we have a surjective map $ W_{q+1} \rightarrow L_0/L $.  The kernel of this map has dimension $ m-r $ and intersects trivially the subspace $ W_q$.  Thus we get $m-r$ dimensional subspace of $ W_{q+1}/W_q = \spn(z^qe_1, \dots, z^qe_m) = \C^m $.  Thus we get a map
$ \X(k)_0 \rightarrow \G(m-r, m) $. This map is clearly $ GL_m $ equivariant and since $ GL_m $ acts transitively on the base, it is a $ GL_m $-fibre bundle.

Let $ \X(k)_1 $ denote the fibre of $ \X(k)_0 \rightarrow \G(m-r, m) $ over the point $ \spn(z^qe_{r+1}, \dots, z^qe_m)$.  We can see that $ \X(k)_1 $ is precisely the locus of points in $ \X(k) $ such that $$ [e_1], \dots, [z^q e_1], \dots, [e_r], \dots, [z^q e_r], [e_{r+1}], \dots, [z^{q-1} e_{r+1}], \dots, [e_m], \dots, [z^{q-1}e_m] $$
 forms a basis for $ L_0 / L $ and $ z^q e_i \in W_q + L_0 $ for $ i = r+1, \dots, m $.

\subsection{A version of the Mirkovi\'c-Vybornov isomorphism}

More generally, for any $ \mu = (\mu_1, \dots, \mu_m) \in \N^m $ we can consider the locus
\begin{equation}
 \begin{aligned}
  \X_\mu = \{  L : &[e_1], \dots, [z^{\mu_1 - 1} e_1], \dots, [e_m], \dots, [z^{\mu_m-1}e_m] \text{ forms a basis for $ L_0 / L $ } \\
& \text{ and } z^{\mu_i} e_i \in W_{\mu_i} + L \text{ for all $i$} \}
 \end{aligned}
\end{equation}
Note that if $\mu_1 + \dots + \mu_m = k $, then $ \X_\mu \subset \X(k) $. Also note that
$$\X_{(q+1, \dots, q+1, q, \dots, q)} = \X(k)_1.$$
Let $ M_\mu $ be the variety of $ k \times k $ block matrices of the following form
\begin{equation}\label{eq:M}
\begin{aligned}
M_\mu = \{ &A = (A_{ij}) : A_{ij} \text{ is a matrix of size $ \mu_i \times \mu_j $ where} \\
& \text{$A_{ii}$ has 1s just below the diagonal and all other non-zero entries in the last column},\\
& \text{$A_{ij}$, for $i \ne j$, has all non-zero entries in the last column but not below row $ \mu_j $} \}
\end{aligned}
\end{equation}
 Note that $ M_\mu $ is an affine space of dimension $ \sum_{i,j} \min(\mu_i, \mu_j) $.
 
 \begin{Example}
 Consider $ \mu = (3,2)$.  Then $ M_\mu $ is the set of all matrices of the form
 $$
 \begin{bmatrix}
 0 & 0 & * & 0 & * \\
  1 & 0 & * & 0 & * \\
    0 & 1 & * & 0 & 0 \\
      0 & 0 & * & 0 & * \\
        0 & 0 & * & 1 & * \\
        \end{bmatrix}
        $$
 where $ * $ denotes an arbitrary complex number.
 \end{Example}

The following isomorphism was first constructed by Mirkovi\'c-Vybornov \cite{MVy} (though we believe our construction is a bit simpler).  The matrices in $M_\mu $ are the transpose of those considered by Mirkovi\'c-Vybornov.

\begin{Theorem} \label{th:MViso}
We have an isomorphism $\X_\mu \cong M_\mu $.
\end{Theorem}

\begin{proof}
We define maps $ \X_\mu \rightarrow M_\mu $ and $ M_\mu \rightarrow \X_\mu $ as follows.

First, given $ L \in \X_\mu$, let $$ B = \{ [e_1], \dots, [z^{\mu_1 - 1}e_1], \dots, [e_m], \dots, [z^{\mu_m - 1}e_m] \} $$ which by hypothesis is a basis for $ L_0/L$.  So we define $ A $ to be the matrix of $ z $ acting on $ L_0 / L $ with respect to the basis $ B $. Note that if $ a < \mu_i $, then $ z[z^{a-1}e_i] = [z^a e_i] $ is still in our basis, whereas $ z[z^{\mu_i -1} e_i] = [z^{\mu_i} e_i] $ can be written as a linear combination of $ [z^r e_j] $ for $ r < \mu_i $ (since $ z^{\mu_i} e_i \in W_{\mu_i} + L $).  This explains why $ A $ lies in $ M_\mu $.

Conversely, given $ A \in M_\mu $, let $ p_{ij}(z) $ be the polynomial made using the entries in the last column $ A_{ij} $ as coefficients.  Then we set $ v_j = z^{\mu_j}e_j - \sum_i p_{ij}(z) e_i $.  We define $ L = \spn(v_1, \dots, v_n) $.

Now we check that these maps are mutual inverses. Begin with $ A \in M_\mu $ and define $ L = \spn(v_1, \dots, v_n) $ as above.  Then $ [z^{\mu_j}e_j] = \sum_i [p_{ij}(z) e_i] $ in $ L_0/L$, and thus the matrix constructed from $ L $ will be $ A $ again.  The converse is similar.
\end{proof}

\subsection{Proof of smoothness}

\begin{Corollary}
$\X(k)_0 $ is smooth.
\end{Corollary}

\begin{proof}
Since $ \X(k)_0 \rightarrow \G(m-r, m) $ is fibre bundle, it is enough to check $\X(k)_1 $ is smooth, which follows from the previous theorem.
\end{proof}

\begin{Proposition}\label{prop:smooth}
The variety $\X(k)$ is smooth.
\end{Proposition}
\begin{proof}
Just like $\Y(k$), $\X(k)$ comes equipped with a map $\ch: \X(k) \rightarrow \A_k = \A^k/S_k$ which records the eigenvalues of $z$ acting $L_0/L$. This map is $\C^\times$ equivariant where $\C^\times$ acts diagonally on $\A^k$. Since the singular locus of $\X(k)$ is closed and $\C^\times$ invariant it suffices to show that $\X(k)$ is smooth at all the points in $\ch^{-1}(0)$.  Also the group $GL_m(\C[z]) $ acts on $ \X(k) $ preserving $ \ch^{-1}(0)$ and the singular locus.

From the previous Corollary we know that the open set $ \X(k)_0 $ is smooth.  Thus, by the above discussion, it suffices to show that $\X(k)_0 $ meets every $ GL_m(\C[z]) $ orbit in $ \ch^{-1}(0)$.

This fiber $\ch^{-1}(0)$ is the closure $\overline{\GrBD^{k\omega_1}} = \overline{GL_m(\C[z]) z^{k\omega_1}} $ of the Schubert cell in the affine Grassmannian corresponding to the weight $k \omega_1$.   On the other hand, $ \X(k)_0 \cap \ch^{-1}(0) = GL_m(\C[z^{-1}])z^{\omega_r} \cap\overline{\GrBD^{k\omega_1}}  $ (where as usual $ k = qm + r $).  Standard result (see for example \cite{BF}) about the geometry of the affine Grassmannian imply that $ GL_m(\C[z^{-1}])z^{\omega_r} $ meets every $ GL_m(\C[z]) $ orbit in this connected component and thus we are done.
\end{proof}

\begin{Corollary} \label{cor:smooth}
 All the varieties $ \Y(\uk) $ are smooth.
\end{Corollary}

\subsection{Relation to the Beilinson-Drinfeld Grassmannian}

The variety $ \X $ defined above can be seen as the positive part of the Beilinson-Drinfeld Grassmannian.  More precisely, we can see that $ \X $ is the space of finite codimension subsheaves of the trivial vector bundle on $ \A^1 $.  Such a subsheaf is automatically a rank $ m $ locally-free sheaf and automatically extends to a rank $m $ vector bundle on $ \P^1$.

Define
\begin{equation*}
 \begin{aligned}
\GrBD_{BD} := \{ (\mathcal V, \sigma) : & \text{ $ \mathcal V $ is a rank $m$ vector bundle on $ \P^1 $ and $\sigma : \mathcal V \dashrightarrow \mathcal O_{\P^1}^m$ } \\
& \text{ is a trivialization defined away from finitely many points in $ \A^1$ } \}.
\end{aligned}
\end{equation*}
We see that $ \X $ can be embedded into $ \GrBD_{BD} $ as the locus where $ \sigma $ extends to an inclusion of coherent sheaves.  Moreover, $ \X_\mu $ is the locus of pairs $ (\mathcal V, \sigma) $ such that $ \mathcal V \cong \O(-\mu_1) \oplus \cdots \oplus \O(-\mu_m) $ and such that the trivialization $ \sigma $ at $ \infty $ takes the Harder-Narasimhan partial flag of $ \mathcal V $ to the standard partial flag.  (In fact $ \X(k)_0 $ is the locus of pairs $ (\mathcal V, \sigma) $ such that $ \mathcal V \cong \O(-q-1)^r \oplus \O(-q)^{m-r} $ --- this is a generic condition for a degree $ -k $ vector bundle.)

\subsection{Factorization and smallness}

Recall that a proper surjective morphism $ \pi : Y \rightarrow X $ is called semismall, if $ X $ admits a finite stratification $ X = \sqcup_{\alpha \in J} X(\alpha) $ such that for each $ \alpha $, and each $ x \in X(\alpha) $, $ \dim \pi^{-1}(x) \le \frac{1}{2} \codim X(\alpha) $.  A semismall morphism is called small, if we have a strict equality for every stratum, except for the generic stratum (in which case $\codim S_\alpha = 0 $).

A factorizable semismall sequence is sequence of commutative diagrams\begin{equation}
\label{eq:factdiagram}
\xymatrix{
\cY_n \ar[rr]^{\pi} \ar[dr]_{f_Y} && \cX_n \ar[dl]^{f_X} \\
& \A^n &
}
\end{equation}
for $n \ge 1$ with $ \cY_n, \cX_n $ irreducible, $ f_X $ flat, and satisfying the following conditions.
\begin{enumerate}
\item \label{item:semismall} Define $ X_n = f_X^{-1}(0)$, $ Y_n =f_Y^{-1}(0) $.  The map $ Y_n \rightarrow X_n $ is semismall.  Moreover, $ Y_1 \rightarrow X_1 $ is an isomorphism.
\item \label{item:diagonal} For each $n $, we have a diagonal copy of $ \A^1 $ inside $ \A^n $.  We should be given isomorphisms $ f_Y^{-1}(\A^1) \cong Y_n \times \A^1 $ and $ f^{-1}(\A^1) \cong X_n \times \A^1 $ such that the restriction of \eqref{eq:factdiagram} to this locus gives the diagram
$$
\xymatrix{
Y_n \times \A^1 \ar[rr]^{(\pi, id)} \ar[dr] && X_n \times \A^1 \ar[dl]\\
& \A^1 &
}
$$
\item \label{item:disjoint} For each $k,l $ with $k+l = n$, let us write $ (\A^k \times \A^l)^\disj $ for the locus $(\ux,\uy)$ such that $x_i \ne y_j$ and denote $ (\cY_k \times \cY_l)^\disj  = (f_Y, f_Y)^{-1} (\A^k \times \A^l)^\disj $ and similarly for $ (\cX_k \times \cX_l)^\disj$. We require isomorphisms $ f_Y^{-1}(\A^k \times \A^l)^\disj \cong (\cY_k \times \cY_l)^\disj $ and $f_X^{-1}(\A^k \times \A^l)^\disj \cong (\cX_k \times \cX_l)^\disj $ such that the restriction of \eqref{eq:factdiagram} to $ (\A^k \times \A^l)^\disj $ gives the diagram
$$
\xymatrix{
(\cY_k \times \cY_l)^\disj \ar[rr]^{(\pi,\pi)} \ar[dr]_{(f_Y,f_Y)} && (\cX_k \times \cX_l)^\disj \ar[dl]^{(f_X,f_X)} \\
&(\A^k \times \A^l)^\disj &
}
$$
\end{enumerate}

\begin{Theorem} \label{th:factsmall}
If $ \cY_n \rightarrow \cX_n $ is a factorizable semismall sequence then $ \cY_n \rightarrow \cX_n $ is small for each $ n \ge 1$.
\end{Theorem}
\begin{proof}
Consider an equivalence relation $\sigma$ on $I_n := \{1, \dots, n\}$ and let $|\sigma|$ be the number of equivalence classes in $\sigma$. Denote by $\Delta_\sigma$ the corresponding diagonal inclusion of $(\A^{|\sigma|})^\disj$ into $\A^n$. For example, for the unique relation $ \sigma $ with one equivalence class, the map $\Delta_\sigma: \A^1 \rightarrow \A^n$ is the small diagonal. While for the trivial relation (where there are $n$ distinct equivalence classes), the image of $\Delta_\sigma$ is the open dense subset $(\A^n)^\disj$. Clearly
$$\A^n = \sqcup_{\sigma} \Delta_\sigma (\A^{|\sigma|})^\disj.$$

Fix an equivalence relation $ \sigma $ on $I_n$.  If we restrict $\pi: \cY_n \rightarrow \cX_n$ to $\Delta_\sigma(\A^{|\sigma|})$ and repeatedly apply relations (\ref{item:disjoint}) and (\ref{item:diagonal}), then we obtain
\begin{equation}
\label{eq:diagonals}
Y_{\lambda_1} \times \cdots \times Y_{\lambda_{|\sigma|}} \times (\A^{|\sigma|})^\disj \rightarrow X_{\lambda_1} \times \cdots \times X_{\lambda_{|\sigma|}} \times (\A^{|\sigma|})^\disj
\end{equation}
where $\l_1, \dots, \l_{|\sigma|}$ are the sizes of the equivalence classes.

We can now describe the stratification for $\cX_n$ as follows. First we have
$$\cX_n = \sqcup_{\sigma} f_X^{-1} (\Delta_\sigma(\A^{|\sigma|})^\disj).$$
Next, from (\ref{eq:diagonals}) above
$$f_X^{-1}(\Delta_\sigma(\A^{|\sigma|})^\disj) \cong X_{\l_1} \times \dots \times X_{\l_{|\sigma|}} \times (\A^{|\sigma|})^\disj.$$
By (\ref{item:semismall}), each $X_j$ comes with a stratification $X_j = \sqcup_{\alpha} X_j(\alpha)$ for which $Y_j \rightarrow X_j$ is semismall. This gives us a stratification
$$X_{\l_1} \times \dots \times X_{\l_{|\sigma|}} \times (\A^{|\sigma|})^\disj = \bigsqcup_{\underline{\alpha}} X_{\l_1}(\alpha_1) \times \dots X_{\l_{|\sigma|}}(\alpha_{|\sigma|}) \times (\A^{|\sigma|})^\disj.$$
Thus we get an overall stratification $ \cX_n = \sqcup_{\sigma,\underline{\alpha}} \cX_n(\sigma, \underline{\alpha}) $ of $ \cX_n$. It remains to check that $\pi: \cY_n \rightarrow \cX_n$ is small for this stratification.

Let us choose a point $x$ in the stratum $\cX(\sigma, \underline{\alpha}) = X_{\l_1}(\alpha _1) \times \dots X_{\l_{|\sigma|}}(\alpha_{|\sigma|}) \times (\A^{|\sigma|})^\disj$. Since each $Y_i \rightarrow X_i$ is semismall and using (\ref{eq:diagonals}) we find that
\begin{equation}
 \label{eq:dimfibre}
 \dim \pi^{-1}(x) \le \frac{1}{2} \sum_i d_{\lambda_i}(\alpha_i)
\end{equation}
where $d_j(\alpha)$ denotes the codimension of $X_j(\alpha)$ inside $X_j$. On the other hand, the codimension of $ \cX_n(\sigma, \underline{\alpha}) $ in $\cX_n$ is
\begin{equation}
 \label{eq:codim}
 \codim \cX_n(\sigma, \underline{\alpha}) = \sum_i d_{\l_i}(\alpha_i) + \dim \cX_n - \sum_i \dim X_{\l_i} - {|\sigma|}
\end{equation}
Now, by flatness and by considering the fibre $ f_X^{-1}(0)$, we see that $\dim \cX_n = n + \dim X_n $.  On the other hand, considering the preimage of the generic stratum in $ \A^n $, we see that $ \dim \cX_n = n \dim \cX_1$.  Thus \eqref{eq:codim} gives
$$ \codim \cX_n(\sigma, \underline{\alpha}) = \sum_i d_{\lambda_i}(\alpha_i) + n - {|\sigma|}.$$
Comparing with \eqref{eq:dimfibre}, we see that $\dim \pi^{-1}(x)$ is strictly less than half the codimension of $\cX_n(\sigma, \underline{\alpha})$ in $\cX_n$ unless ${|\sigma|}=n$ which corresponds to the open strata. The result follows.
\end{proof}

With this general framework in mind, we now return to our situation. First of all, there is a map $ \Y(1^n) \rightarrow \Y(n)$, given by $ \{ L_0 \xrightarrow{1} L_1 \xrightarrow{1} \cdots \xrightarrow{1} L_n \} \mapsto L_n$.  Together with $ \ch : \Y(1^n) \rightarrow \A^n$ we get a map $ \pi :\Y(1^n) \rightarrow \Y(n) \times_{\A_n} \A^n $.

\begin{Proposition} \label{prop:isfact}
The sequence of diagrams
$$
\xymatrix{
\Y(1^n) \ar[rr]^\pi \ar[dr]^\ch && \Y(n) \times_{\A_n} \A^n \ar[dl] \\
&\A^n &
}
$$
is a factorizable semismall sequence.
\end{Proposition}

\begin{proof}
We must check the three properties in the definition.

For the first property (\ref{item:semismall}), let $ V_n = \ch^{-1}(0) \subset \Y(1^n) $ and $ W_n = \ch^{-1}(0) \subset \Y(n) $.  Note that $ W_n $ is also the fibre over $ 0 $ in the map $  \Y(n) \times_{\A_n} \A^n  \rightarrow \A^n $.  We can see that $ W_n $ is just the locus $\overline{\GrBD^{n\omega_1}} $ in the affine Grassmannian of $ GL_m$ and $ V_n $ is just its usual resolution using an iterated bundle of $ \GrBD^{\omega_1} $ (in fact $ V_n = Y(1^n) $).  This is well-known to be semismall.  Finally, we see that $ V_1 = W_1 $.

For the second property (\ref{item:diagonal}), given $ x \in \C $ and $ L \in \Y(n) $ such that $ \ch(L) = x^n $, we construct $ L_{\C[[z-x]]} = L \otimes_{\C[z]} \C[[z-x]] $.  Then we see that $ \dim L_{\C[[z-x]]} / \C[[z-x]]^m = n $.  Using the obvious isomorphism $\C[[z-x]] \cong \C[[z]] $, we can regard $  L_{\C[[z-x]]}  $ as a point of $ W_n $.  Thus, we define the map $ \ch^{-1}(\A^1) \rightarrow W_n \times \A^1 $ by $ L \mapsto  (L_{\C[[z-x]]},x) $.  It is easy to see that this is an isomorphism.  A similar analysis applies for $ \Y(1^n) $.

The third property (\ref{item:disjoint}) is an application of the standard factorization property of the Beilinson-Drinfeld Grassmannian \cite[section 5.3.12]{BD}.  Let $ \Y(n)_{k,l}^\disj$ be the fibre product of $( \A^k \times \A^l )^\disj$ and $ \Y(n) $ over $ \A_n $.  We can describe this space as
$$ \Y(n)_{k,l}^\disj = \{ \C[z]^m \subset L, (\ux,\uy) : \{\ux, \uy\} = \ch(L) \}$$
We define a map
\begin{align*}
\Y(n)_{k,l}^\disj  &\rightarrow (\Y(k)\times_{\A_k} \A^k \times \Y(l) \times_{\A_l} \A^l )^\disj \\
(L, (\ux,\uy)) &\mapsto ((L_1, \ux), (L_2, \uy))
\end{align*}
where $ L_1 \in \Y(k) $ is defined by the property that $L_1/\C[z]^m $ is the direct sum, over $ i = 1, \dots, k$, of the generalized $x_i$-eigenspaces of $ z|_{L/\C[z]^m} $, and $L_2 $ is defined in a similar way using $ \uy $.
This map is an isomorphism.
\end{proof}

\begin{Corollary} \label{cor:small}
For any $\uk$ the projection map $ \Y(\uk) \rightarrow \Y(k_1 + \dots + k_n) $ is small.
\end{Corollary}

\begin{proof}
When all $ k_i = 1 $ we use Theorem \ref{th:factsmall} and Proposition \ref{prop:isfact} to conclude that $ \Y(1^n) \rightarrow \Y(n) \times_{\A_n} \A^n $ is small.  Since the composition of a small morphism and a finite morphism is small, we conclude that the morphism $ \Y(1^n) \rightarrow \Y(n) $ is small.

For a general $\uk$ consider the composition
$$ \Y(1^{k_1},\dots,1^{k_n}) \rightarrow \Y(k_1, \dots, k_n) \rightarrow \Y(k_1 + \dots + k_n). $$
Since the composition is small it follows that the right hand map is small.
\end{proof}

\section{Categorical actions}\label{sec:catactions}

In this section we define and discuss the notion of an $(L\gl_n,\t)$ action. This notion formalizes what it means to have the quantum affine algebra $U_q(\hgl_n)$ acting (at level zero) on a category. We write $L\gl_n$ instead of $L\sl_n$ just because it is convenient to label the weight spaces with $\uk \in \bZ^n$, which is identified with the weight lattice of $\gl_n$.

For $n \ge 1$ we denote by $[n]$ the quantum integer $q^{n-1} + q^{n-3} + \dots + q^{-n+3} + q^{-n+1}$. By convention $[-n] = -[n]$. If $f = \sum_a f_a q^a \in \N[q,q^{-1}]$ and $\sA$ is a 1-morphism inside a graded 2-category we write $\bigoplus_f \sA$ for the direct sum $\bigoplus_{a \in \bZ} \sA^{\oplus f_a} \la a \ra$. For example,
$$\bigoplus_{[n]} \sA = \sA \la n-1 \ra \oplus \sA \la n-3 \ra \oplus \dots \oplus \sA \la -n+3 \ra \oplus \sA \la -n+1 \ra.$$

We identify the weight lattice of $\gl_n$ with $\bZ^n$ so that $\alpha_i = (0,\dots,-1,1,\dots,0)$ for $i \in I = \{1,\dots,n-1\}$ generates the root lattice. We write $\alpha_0 := - \alpha_1 - \dots - \alpha_{n-1}$. We also denote by $\la \cdot, \cdot \ra: \bZ^n \times \bZ^n \rightarrow \Z$ the standard pairing. For convenience we write $\la i,j \ra$ for $\la \alpha_i, \alpha_j \ra$.

An $(L\gl_n,\t)$ action consists of a target 2-category $\K$ which is graded, triangulated, $\C$-linear and idempotent complete. The objects $\K(\uk)$ in $\K$ are indexed by sequences $\uk \in \bZ^n$ and equipped with
\begin{enumerate}
\item 1-morphisms: $\E_{i,\ell} \1_\uk = \1_{\uk+\alpha_i} \E_{i,\ell}$ and $\F_{i,-\ell} \1_{\uk+\alpha_i} = \1_\uk \F_{i,-\ell}$ where $i \in I$, $\ell \in \{0,1\}$ and $\1_\uk$ is the identity 1-morphism of $\K(\uk)$.
\item 2-morphisms: for each $\uk \in \bZ^n$ a linear map
$$\spn \{\alpha_i: i \in I\} \rightarrow \End^2(\1_\uk).$$
\end{enumerate}

On this data we impose the following conditions.
\begin{enumerate}
\item \label{co:hom1} $\Hom(\1_\uk, \1_\uk \la l \ra)$ is zero if $l < 0$ and one-dimensional if $l=0$ and $\1_\uk \ne 0$. Moreover, the space of maps between any two 1-morphisms is finite dimensional.
\item \label{co:adj} $\E_{i,\ell}$ and $\F_{i,-\ell}$ are left and right adjoints of each other up to specified shifts. More precisely
\begin{enumerate}
\item $(\E_{i,\ell} \1_\uk)^R \cong \1_\uk \F_{i,-\ell} \la \la \uk, \alpha_i \ra + 1 \ra$
\item $(\E_{i,\ell} \1_\uk)^L \cong \1_\uk \F_{i,-\ell} \la - \la \uk, \alpha_i \ra - 1 \ra$.
\end{enumerate}

\item \label{co:EF} For $i \in I$ and $\ell \in \{0,1\}$ we have
\begin{align*}
\E_{i,\ell} \F_{i,-\ell} \1_\uk &\cong \F_{i,-\ell} \E_{i,\ell} \1_{\uk} \bigoplus_{[\la \uk, \alpha_i \ra]} \1_\uk \ \ \text{ if } \la \uk, \alpha_i \ra \ge 0 \\
\F_{i,-\ell} \E_{i,\ell} \1_{\uk} &\cong \E_{i,\ell} \F_{i,-\ell} \1_{\uk} \bigoplus_{[- \la \uk ,\alpha_i \ra]} \1_\uk \ \ \text{ if } \la \uk, \alpha_i \ra \le 0
\end{align*}

\item \label{co:EiFj} If $i \ne j \in I$ then $\F_{j,-\ell'} \E_{i,\ell} \1_\uk \cong \E_{i,\ell} \F_{j,\ell'} \1_\uk$ for any $\ell,\ell' \in \{0,1\}$.

\item \label{co:theta}
For $i \in I$, $\ell \in \{0,1\}$ we have
$$\E_{i,\ell} \E_{i,\ell} \cong \E_{i,\ell}^{(2)} \la -1 \ra \oplus \E_{i,\ell}^{(2)} \la 1 \ra$$
for some 1-morphism $\E_{i,\ell}^{(2)}$. Moreover, if $\theta \in \spn \{\alpha_i: i \in I\}$ then $I \t I \in \End^2(\E_{i,\ell} \1_\uk \E_{i,\ell})$ induce a map between the summands $\E_{i,\ell}^{(2)} \la 1 \ra$ on either side which is
\begin{itemize}
\item nonzero if $\la \theta, \alpha_i \ra \ne 0$ and
\item zero if $\la \theta, \alpha_i \ra = 0$.
\end{itemize}

\item \label{co:triangle}
From the relations above it follows that if $\la i,j \ra = -1$ then
$$\Hom(\E_{i} \E_{j,1} \la -1 \ra \1_\uk, \E_{j,1} \E_{i} \1_\uk) \text{ and }
\Hom(\E_{j} \E_{i,1} \la -1 \ra \1_\uk, \E_{i,1} \E_{j} \1_\uk)$$
are both one-dimensional (or both zero). If $\alpha$ and $\beta$ are 2-morphisms which span these spaces, then we require the relation $\Cone(\alpha) \cong \Cone(\beta)$.

\item \label{co:vanish1} If $\alpha = \alpha_i$ or $\alpha = \alpha_i + \alpha_j$ for some $i,j \in I$ with $\la \alpha_i,\alpha_j \ra = -1$ then $\1_{\uk+r \alpha} = 0$ for $r \gg 0$ or $r \ll 0$.

\item \label{co:new} Suppose $i \ne j \in I$. If $\1_{\uk+\alpha_i}$ and $\1_{\uk+\alpha_j}$ are nonzero then $\1_\uk$ and $\1_{\uk+\alpha_i+\alpha_j}$ are also nonzero.

\end{enumerate}

\subsection{Some remarks on this definition}

\begin{itemize}
\item The concept of a $(\g,\t)$ action was introduced in \cite{C2} in order to give a simplified version of some of the earlier definitions from \cite{CR,KL,Rou}. The $(L\gl_n,\t)$ action in the current paper extends the definition from \cite{C2} to the affine case. It coincides with that appearing in \cite{CL} except that it appears at level zero rather than level one.
\item $\K$ being triangulated means that $\Hom(\K(\uk),\K(\uk'))$ is a triangulated category for any two $\uk,\uk' \in \bZ^n$. Graded means that the 1-morphisms are equipped with an auto-equivalence $\la 1 \ra$ (which should not be confused with the cohomological shift $[1]$). The grading shift $\la 1 \ra$ corresponds to multiplication by $q$ at the level of the quantum group. Moreover, $\K$ is idempotent complete if for any 2-morphism $f$ with $f^2=f$ the image of $f$ is contained in $\K$.
\item The reason we write $\alpha_0$ for $- \sum_{i=1}^{n-1} \alpha_i$ is because we are working with a level zero representation which means that $\la \uk, \alpha_0 \ra = \la \uk, - \sum_{i=1}^{n-1} \alpha_i \ra$ for any $\uk$.
\item For $\t \in \spn \{\alpha_i: i \in I\}$ we abuse notation and denote by $\t \in \End^2(\1_\uk)$ for its image under the map $\spn \{\alpha_i: i \in I\} \rightarrow \End^2(\1_\uk)$.
\item Condition (\ref{co:theta}) corresponds to the alternative (equivalent) definition of an $(\sl_n,\t)$ action from \cite[Section 13]{C2}. We prefer this definition because it is easier to check.
\item In the remainder of this paper we will have that $\K(\uk)$ is zero (i.e. $\1_\uk=0$) if and only if some $k_i < 0$. Subsequently, conditions (\ref{co:vanish1}) and (\ref{co:new}) are trivial to check.
\end{itemize}

\subsection{Affine braid group actions}\label{sec:braids}

We now explain how such an action gives rise to an action of the affine braid groupoid $\ABr_n$ on $\K$. Recall that $\ABr_n$ has objects indexed by $\bZ^n$ and morphisms
\begin{itemize}
\item $T_i \in \Hom(\uk, s_i \cdot \uk)$ for $i \in I$,
\item $\phi_\alpha \in \End(\uk)$ for $\alpha = \sum_{i \in I} a_i \alpha_i$ an element of the root lattice.
\end{itemize}
subject to the following relations
\begin{itemize}
\item $T_iT_j \cong T_jT_i$ if $\la i,j \ra = 0$,
\item $T_iT_jT_i \cong T_jT_iT_j$ if $\la i,j \ra = -1$ (the braid relation),
\item $\phi_{\alpha} \phi_{\beta} \cong \phi_\beta \phi_\alpha$ for any $\alpha,\beta$ in the root lattice,
\item $T_i \phi_{\alpha} \cong \phi_{\alpha} T_i$ if $\la \alpha_i, \alpha \ra = 0$,
\item $T_i \phi_\alpha T_i \cong \phi_{s_i \cdot \alpha}$ if $\la \alpha_i, \alpha \ra = -1$.
\end{itemize}

Consider first an $(\sl_2,\t)$ action generated by $\E$ and $\F$. Then one can define the Rickard complexes
\begin{align*}
[\dots \rightarrow \E^{(2)} \F^{(\l+2)} \la -2 \ra \rightarrow \E \F^{(\l+1)} \la -1 \ra \rightarrow \F^{(\l)}] & \text{ if } \l \ge 0, \\
[\dots \rightarrow \F^{(2)} \E^{(-\l+2)} \la -2 \ra \rightarrow \F \E^{(-\l+1)} \la -1 \ra \rightarrow \E^{(-\l)}] [-\l] \la \l \ra & \text{ if } \l \le 0
\end{align*}
to give us 1-morphisms in $\Hom(\K(k,l),\K(l,k))$, where $\l := -k+l$. Following \cite{CR,CKL2} one can show that these if $\K$ is triangulated then these complexes have a unique convolution $\T$ and that this $\T$ is invertible.

\begin{Remark}
The definition above differs from the ones in \cite{CR,CKL2} by the extra shift $[-\l] \la \l \ra$ when $\l \le 0$. This is for convenience since it will simplify some subsequent relations.
\end{Remark}

\begin{Remark}\label{rem:1}
The differentials in the Rickard complexes can be made explicit. Conveniently this is not entirely necessary since one can show \cite[Lemma 5.9]{CK3} that 
$$\Hom(\E^{(s)} \F^{(\l+s)} \la -1 \ra, \E^{(s-1)} \F^{(\l+s-1)}) \ \ \text{ and } \ \ \Hom(\E^{(s-1)} \F^{(\l+s-1)} \la -1 \ra, \E^{(s)} \F^{(\l+s)})$$
are both one-dimensional. This implies that the nonzero differentials in the Rickard complexes are unique (up to rescaling). Since rescaling leads to homotopic complexes this means that the Rickard complexes above are unique (up to isomorphism) as long as you assume the differentials are nonzero.
\end{Remark} 

Given an $(L\gl_n,\t)$ action as above we have two $\sl_2$ actions for each $i \in I$, one generated by $\E_i,\F_i$ and one by $\E_{i,1},\F_{i,-1}$. We denote the resulting equivalences $\T_i$ and $\T_{i,1}$ respectively. Following \cite{CK3} it follows that $\T_i \T_j \cong \T_j \T_i$ if $\la i,j \ra = 0$ and $\T_i \T_j \T_i \cong \T_j \T_i \T_j$ if $\la i,j \ra = -1$ (and likewise if we replace with $\T_{i,1}$ or $\T_{j,1}$).

Consider $i,j \in I$ with $\la i,j \ra = -1$. Using \cite[Lemma A.4]{C2} we have that
$$\dim \Hom(\E_i \E_{j,1} \la -1 \ra \1_\uk, \E_{j,1} \E_i \1_\uk) \le 1$$
with equality if both 1-morphisms are nonzero. We denote this nonzero map $\alpha$. This explains the existence of the first map in condition (\ref{co:triangle}). Similarly, we have a map
$$\beta: \E_j \E_{i,1} \la -1 \ra \1_\uk \rightarrow \E_{i,1} \E_j \1_\uk.$$

\begin{Remark}
Alternatively one can think of $\alpha$ as one of the generators of the quiver Hecke algebra action constructed in \cite{C2} for the $(\sl_3,\t)$ action generated by $\E_i,\F_i,\E_{j,1},\F_{j,-1}$ (and likewise for $\beta$). However, this extra structure will not be necessary for our purposes.
\end{Remark}

\begin{Lemma}\label{lem:conjugation}
If $\la i,j \ra = -1$ then $\T_i \E_{j,1} \T_i^{-1} \cong \T_j \E_{i,1} \T_j^{-1}$.
\end{Lemma}
\begin{proof}
Replacing $\E_j$ with $\E_{j,1}$ in \cite[Lemma 5.2]{C1} we find that
\begin{equation}\label{eq:4}
\Cone \left( \E_i \E_{j,1} \la -1 \ra \xrightarrow{\alpha} \E_{j,1} \E_i \right) \T_i \cong \T_i \E_{j,1}
\end{equation}
where the complex on the left hand side is in cohomological degrees $-1$ and $0$. Likewise, exchanging $i$ and $j$ in (\ref{eq:4}) we get
\begin{equation}\label{eq:5}
\Cone \left(\E_j \E_{i,1} \la -1 \ra \xrightarrow{\beta} \E_{i,1} \E_j \right) \T_j \cong \T_j \E_{i,1}.
\end{equation}
Putting this together we have
\begin{align*}
\T_i \E_{j,1} \T_i^{-1}
&\cong \Cone \left(\E_i \E_{j,1} \la -1 \ra \xrightarrow{\alpha} \E_{j,1} \E_i \right) \\
&\cong \Cone \left(\E_j \E_{i,1} \la -1 \ra \xrightarrow{\beta} \E_{i,1} \E_j \right) \\
&\cong \T_j \E_{i,1} \T_j^{-1}
\end{align*}
where we used condition (\ref{co:triangle}) to get the second isomorphism.
\end{proof}

\begin{Corollary}\label{cor:conjugation}
If $\la i,j \ra = -1$ then
$$\T_i \T_{j,1} \T_i^{-1} \cong \T_j \T_{i,1} \T_j^{-1}, \ \ \text{ and } \ \ \T_{i,1} \T_i \T_{j,1} \cong \T_j \T_{i,1} \T_i.$$
\end{Corollary}
\begin{proof}
Applying the isomorphism in Lemma \ref{lem:conjugation} repeatedly we find that
$$\T_i \E_{j,1}^{(r)} \T_i^{-1} \cong \T_j \E_{i,1}^{(r)} \T_j^{-1}$$
for any $r \in \N$. Taking adjoints we also get $\T_i \F_{j,-1}^{(r)} \T_i^{-1} \cong \T_j \F_{i,-1}^{(r)} \T_j^{-1}$. Thus
$$\T_i \F_{j,-1}^{(t+s)} \E_{j,1}^{(s)} \T_i^{-1} \cong \T_j \F_{i,-1}^{(t+s)} \E_{i,1}^{(s)} \T_j^{-1}$$
for any $s,t$. Since the differentials in the Rickard complexes are unique up to rescaling (see Remark \ref{rem:1}) this implies that $\T_i \T_{j,1} \T_i^{-1} \cong \T_j \T_{i,1} \T_j^{-1}$ (there is a minor issue of the $[-\l]\la \l \ra$ shift in the definition of $\T_{j,1}$ and $\T_{i,1}$ but it is easy to see the shifts agree since $\la s_i \cdot \uk, \alpha_j \ra = \la s_j \cdot \uk, \alpha_i \ra$ for any $\uk \in \bZ^n$). This proves the first isomorphism.

Now, $\T_{i,1}$ and $\T_j$ satisfy the braid relation $\T_{i,1} \T_j \T_{i,1} \cong \T_j \T_{i,1} \T_j$ for the same reason that $\T_i$ and $\T_j$ do. Thus $\T_j \T_{i,1} \T_j^{-1} \cong \T_{i,1}^{-1} \T_j \T_{i,1}$ which means $\T_i \T_{j,1} \T_i^{-1} \cong \T_{i,1}^{-1} \T_j \T_{i,1}$. This proves the second isomorphism.
\end{proof}

\begin{Proposition}\label{prop:phi}
For $i \in I$ define $\phi_i := \T_{i,1}\T_i$. Then $\phi_i \phi_j \cong \phi_j \phi_i$.
\end{Proposition}
\begin{proof}
If $\la i,j \ra = 0$ this is clear since $\E_i,\E_{i,1},\F_i,\F_{i,1}$ all commute with $\E_j,\E_{j,1},\F_j,\F_{j,1}$. If $\la i,j \ra = -1$ then we have
\begin{align*}
\phi_i \phi_j
&\cong \T_{i,1} \T_i \T_{j,1} \T_j \\
&\cong \T_j \T_{i,1} \T_i \T_j \\
&\cong \T_j \T_{i,1} \T_j^{-1} \T_i \T_j \T_i \\
&\cong \T_i \T_{j,1} \T_i^{-1} \T_i \T_j \T_i \\
&\cong \T_{j,1} \T_j \T_{i,1} \T_i \\
&\cong \phi_j \phi_i
\end{align*} where we used Corollary \ref{cor:conjugation} to get the second, fourth and fifth isomorphisms.
\end{proof}

For $\alpha = \sum_{i \in I} a_i \alpha_i \in Y$ we define
$$\phi_{\alpha} := \prod_{i \in I} \phi_i^{a_i}.$$
By Proposition \ref{prop:phi} this is well defined.

\begin{Theorem}\label{thm:affinebraid}
The $\T_i, \phi_{\alpha}$ defined above give rise to an affine braid group action.
\end{Theorem}
\begin{proof}
First note that $\T_j \phi_i \cong \phi_i \T_j$ if $\la i,j \ra = 0$ for obvious reasons. Moreover, using Corollary \ref{cor:conjugation}, if $\la i,j \ra = -1$ then
$$\T_j \phi_i \T_j \cong \T_j \T_{i,1} \T_i \T_j \cong \T_{i,1} \T_i \T_{j,1} \T_j \cong \phi_i \phi_j.$$
This also implies that
$$\T_j \phi_i^2 \phi_j \cong \T_j \phi_i \T_j \phi_i \T_j \cong \phi_i^2 \phi_j \T_j.$$
It is an easy exercise to see that this implies $\T_j \phi_\alpha \cong \phi_\alpha \T_j$ if $\la \alpha, \alpha_j \ra = 0$. Moreover, if $\la \alpha, \alpha_j \ra = -1$ then $$\T_j \phi_\alpha \T_j \cong \T_j \phi_{\alpha-\alpha_i} \phi_i \T_j \cong \phi_{\alpha-\alpha_i} \T_j \phi_i \T_j \cong \phi_{\alpha+\alpha_j} \cong \phi_{s_j \cdot \alpha}$$
where the second isomorphism follows since $\la \alpha-\alpha_i, \alpha_j \ra = 0$.
\end{proof}

\begin{Remark}
Note that one gets a finite braid group action on both $\Kom(\K)$ and $\K$. However, this action extends to an affine one only on $\K$. In particular, in the proof above one needs to apply relation (\ref{co:triangle}) which requires using the triangulated structure of $\K$ and not just working in the homotopy category.
\end{Remark}

The following result which is conceptually helpful.

\begin{Lemma}\label{lem:phiE}
If $\la \alpha, \alpha_i \ra = 0$ then $\E_i \phi_\alpha \cong \phi_\alpha \E_i$. If $\la \alpha, \alpha_i \ra = -1$ then $\phi_\alpha \E_{i,1} \cong \E_i \phi_\alpha$.
\end{Lemma}
\begin{proof}
If $\la i,j \ra = 0$ then $\E_i$ commutes with $\phi_j$ for obvious reasons.

Now suppose $\la i,j \ra = -1$. Then by Lemma \ref{lem:conjugation} we have
$$\phi_j \E_{i,1} \phi_j^{-1} \cong \T_{j,1} \T_j \E_{i,1} \T_j^{-1} \T_{j,1}^{-1} \cong \T_{j,1} \T_i \E_{j,1} \T_i^{-1} \T_{j,1}^{-1}.$$
The right most term above equals $\E_i$ by applying \cite[Lemma 5.2]{C1} to $\{ \E_i, \E_{j,1} \}$ which forms an $\sl_3$ pair. So we get $\phi_j \E_{i,1} \cong \E_i \phi_j$.

On the other hand, still assuming $\la i,j \ra = -1$, we get
\begin{align*}
\E_i \phi_j^2 \phi_i &\cong \phi_j \E_{i,1} \phi_j \phi_i \cong \phi_j \E_{i,1} \T_j \phi_i \T_j \cong \phi_j \E_{i,1} \T_j \T_{i,1} \T_i \T_j \\
& \cong \phi_j \T_j \T_{i,1} \E_j \T_i \T_j \cong \phi_j \T_j \T_{i,1} \T_i \T_j \E_i \cong \phi_j^2 \phi_i \E_i.
\end{align*}
The result follows.
\end{proof}

Finally, we define the shifted complexes $\T_i' \1_\uk := \T_i \1_\uk [-k_i] \la k_i \ra$. These complexes have the nice property that they are shifted precisely so that they lie in positive cohomological degrees, starting from degree zero. They also appeared, for instance, in \cite{C1} for the purposes of defining homological knot invariants. It is easy to check that they still braid. They will only play a role in Sections \ref{sec:shift}, \ref{sec:equiv} and \ref{sec:skewhowe}. One also defines the corresponding $\phi'_i := \T'_{i,1} \T'_i$.

\section{The main result}\label{sec:main}

We take $\tK^{n,N}_{\GrBD,m}$ to be the triangulated 2-category whose nonzero objects are indexed by sequences $\uk \in \bZ^n$ with $\sum_i k_i = N$, the 1-morphisms are kernels inside $D(\Y(\uk) \times \Y(\uk'))$ and the 2-morphisms are maps between kernels.

For $i \in I$, $\ell \in \{0,1\}$ we define kernels
\begin{align*}
\sE_{i,\ell} \1_\uk &:= \O_{\Y^1_i(\uk)} \otimes \det(\sV_i/\sV_i')^{\otimes -\ell} [i \ell] \{k_i-1-i \ell\} \in D(\Y(\uk) \times \Y(\uk+\alpha_i)) \\
\1_\uk \sF_{i,-\ell} &:= \O_{\Y^1_i(\uk)} \otimes \det(\sV'_i/\sV_i)^{\otimes \ell} [-i \ell] \{k_{i+1}+i \ell \} \in D(\Y(\uk+\alpha_i) \times \Y(\uk))
\end{align*}
To define the linear map $\spn \{\alpha_i: i \in I\} \rightarrow \End^2(\1_\uk)$ recall that we have the map
$$\ch: \Y(\uk) \rightarrow \A_{k_1} \times \dots \times \A_{k_n}$$
where $\A_{k_i} = \Spec \C[e_1^{(i)}, \dots, e_{k_i}^{(i)}]$. Then $\alpha_i$ acts by multiplication by $- e_1^{(i)} + e_1^{(i+1)}$ and extend this action linearly to all of $\spn \{\alpha_i: i \in I\}$.

\begin{Theorem}\label{thm:main}
The data above defines an $(L\gl_n,\theta)$ action on $\tK^{n,N}_{\GrBD,m}$.
\end{Theorem}

In the remainder of this section we will prove this Theorem.

\subsection{Some preliminaries}

\begin{Lemma}\label{lem:dualizing}
Consider the projection
$$\pi: \Y(\dots, k_i,k_{i+1}, \dots) \rightarrow \Y(\dots,k_i+k_{i+1},\dots)$$
which forgets $L_i$. Then the relative dualizing sheaf of $\pi$ is $\omega_\pi \cong \O_{\Y(\uk)} \{2k_i k_{i+1}\}$.
\end{Lemma}
\begin{proof}
To simplify notation we assume $\pi: \Y(k,l) \rightarrow \Y(k+l)$. Let $ U $ denote the locus where $ z \in \End(L_2/L_0)$ is regular (it has only one Jordan block for each eigenvalue).  Since a regular linear operator has only finitely many invariant subspaces, the restriction $ \pi_U : U \rightarrow \Y(k+l) $ is finite.

Let us call $D$ the ramification locus in $U$ (this is a divisor). Since $\Y(k,l)$ and $\Y(k+l)$ are smooth (by Corollary \ref{cor:smooth}) the relative dualizing sheaf is a line bundle which is uniquely determined by its restriction to $U$ (here we use that the complement of $U$ has codimension $\ge 2$ by Corollary \ref{cor:small}). Since $\pi_U$ is finite the relative dualizing sheaf is $\O_U(D)$. Hence $\omega_\pi \cong \O_{\Y(k,l)}(\oD)$.

Now consider the commutative diagram
$$\xymatrix{
\Y(k,l) \ar[r]^{\pi} \ar[d]^{\ch} & \Y(k+l) \ar[d]^{\ch} \\
\A_k \times \A_l \ar[r]^{p} & \A_{k+l}
}$$
where $p$ is the natural projection $(\A^k/S_k) \times (\A^l/S_l) \rightarrow (\A^{k+l})/S_{k+l}$. Then $\O_{\Y(k,l)}(\oD) \cong \ch^*(\O_{\A_k \times \A_l}(E))$ where $E$ is the ramification divisor of $p$. On the other hand,
$$\O_{\A_k \times \A_l}(E) \cong \omega_p \cong \omega_{\A_k \times \A_l} \otimes p^* (\omega_{\A_{k+l}}^{-1}).$$
The result follows since $\omega_{\A_k} \cong \O_{\A_k} \{-k(k+1)\}$ where we use that $\A_k \cong \Spec \C[e_1, \dots, e_k]$ where $e_i$ has weight $2i$.

\end{proof}

\begin{Corollary}
Using the same notation as in Lemma \ref{lem:dualizing}, $\pi_* \O_{\Y(\uk)}$ has no higher cohomology.
\end{Corollary}
\begin{proof}
By Lemma \ref{lem:dualizing} it suffices to show that $\pi_* \omega_{\Y(\uk)}$ has no higher cohomology. But this is the same as showing that $H^i(\pi_* \omega_{\Y(\uk)} \otimes L) = 0$ for any very ample line bundle $L$ and $i > 0$. But
$$H^i(\pi_* \omega_{\Y(\uk)} \otimes L) \cong H^i(\omega_{\Y(\uk)} \otimes \pi^* L)$$
which, by Kawamata-Viehweg vanishing theorem, is zero for $i > 0$ since $\pi^* L$ is nef and big (this is where we use that $\pi$ is generically finite).
\end{proof}

\begin{Proposition}\label{prop:main}
Consider the same notation as in Lemma \ref{lem:dualizing} but with $k_{i+1}=1$. Then we have
$$\pi_* \O_{\Y(\uk)} \cong \bigoplus_{[k_i+1]} \O_{\Y(\dots,k_i+1,\dots)} \{-k_i\}.$$
Moreover, if $\Y(\uk) \rightarrow \A^1 = \Spec \C[x]$ is the map which records the eigenvalue of $z$ on $L_{i+2}/L_{i+1}$ then $x \in \Hom(\O_{\Y(\uk)}, \O_{\Y(\uk)} \{2\})$ induces a map $x: \pi_* \O_{\Y(\uk)} \rightarrow \pi_* \O_{\Y(\uk)} \{2\}$ which is an isomorphism on $k_i$ of the summands. The analogous result also holds if $k_i=1$.
\end{Proposition}
\begin{proof}
To simplify notation we assume that all the other $k_j$ are zero so that the map is $\pi: \Y(k,1) \rightarrow \Y(k+1)$. Consider the following commutative diagram
$$\xymatrix{
\Y(k,1) \ar[r]^{\pi_1} \ar[rd]_{\ch} & \Y'(k,1) \ar[d]^{\ch'} \ar[r]^{\pi_2} & \Y(k+1) \ar[d]^{\ch} \\
& \A_k \times \A_1 \ar[r]^{p} & \A_{k+1}
}$$
where $\Y'(k,1)$ is the fiber product $\Y(k+1) \times_{\A_{k+1}} \A_k \times \A_1$. Notice that $\pi = \pi_2 \circ \pi_1$. Since $p$ is flat we have
\begin{align*}
\pi_{2*} \ch'^* (\O_{\A_k \times \A_1})
&\cong \ch^* p_* (\O_{\A_k \times \A_1}) \\
&\cong \ch^* (\bigoplus_{[k+1]} \O_{\A_{k+1}} \{-k\}) \\
&\cong \bigoplus_{[k+1]} \O_{\Y(k+1)} \{-k\}
\end{align*}
where the second isomorphism uses the standard fact that $\C[x,e_1,\dots,e_k]$ is a free $\C[e_1,\dots,e_{k+1}]$-module of rank $k+1$ generated by $1,x,\dots,x^k$ (where we identify $\A_k \cong \Spec \C[e_1,\dots,e_k]$, $\A_1 = \Spec \C[x]$ and $\A_{k+1} = \Spec \C[e_1,\dots,e_{k+1}]$). Note that, in particular, $x$ induces an isomorphism between $k$ of the summands here.

It remains to show that $\pi_* \O_{\Y(k,1)} \cong \pi_{2*} \O_{\Y'(k,1)}$. Since $\pi_1$ is a birational map we have $\pi_{1*} \O_{\Y(k,1)} \cong \tilde{\O}_{\Y'(k,1)}$, the normalization of the structure sheaf of $\Y'(k,1)$ (a priori we do not know it is normal). Now we have the exact triangle
$$\O_{\Y'(k,1)} \rightarrow \tilde{\O}_{\Y'(k,1)} \rightarrow Q$$
where $Q$ is the quotient. Pushing forward by $\pi_2$ we get
\begin{equation}\label{eq:1}
\bigoplus_{[k+1]} \O_{\Y(k+1)} \{-k\} \rightarrow \pi_* \O_{\Y(k,1)} \rightarrow \pi_{2*} Q.
\end{equation}
It remains to show $\pi_{2*} Q = 0$. To do this apply Verdier duality functor $\D$ to the sequence. We have that $\D(\O_{\Y(k+1)}) \cong \omega_{\Y(k+1)} [d]$, where $d=\dim \Y(k+1)=\dim \Y(k,1)$ and that
$$\D \pi_* \O_{\Y(k,1)} \cong \pi_* \D \O_{\Y(k,1)} \cong \pi_* \omega_{\Y(k,1)} [d] \cong \pi_* \O_{\Y(k,1)} \otimes \omega_{\Y(k+1)} \{2k\} [d]$$
where we used Lemma \ref{lem:dualizing} to get the last isomorphism. Thus, after tensoring with $\omega_{\Y(k+1)}^{-1} [-d]\{-2k\}$, we get the exact triangle
\begin{equation}\label{eq:2}
\D \pi_{2*} Q \otimes \omega_{\Y(k+1)}^{-1} [-d] \rightarrow \pi_* \O_{\Y(k,1)} \rightarrow \bigoplus_{[k+1]} \O_{\Y(k+1)} \{-k\}.
\end{equation}
Now consider the composition
$$\bigoplus_{[k+1]} \O_{\Y(k+1)} \{-k\} \rightarrow \pi_* \O_{\Y(k,1)} \rightarrow \bigoplus_{[k+1]} \O_{\Y(k+1)} \{-k\}$$
using the maps from (\ref{eq:1}) and (\ref{eq:2}). These maps are injective so the composition is injective. On the other hand,
$$\Hom(\O_{\Y(k+1)}, \O_{\Y(k+1)} \{i\}) \cong \begin{cases} \C & \text{ if } i = 0 \\ 0 & \text{ if } i < 0 \end{cases}$$
which implies that $\End(\oplus_{[k+1]} \O_{\Y(k+1)})$ is an upper triangular matrix with multiples of the identity map on the diagonal. Together with injectivity this implies that the composition above is an isomorphism. Thus (\ref{eq:1}) splits and we have
$$\pi_* \O_{\Y(k,1)} \cong \bigoplus_{[k+1]} \O_{\Y(k+1)} \{-k\} \oplus \pi_{2*} Q.$$
But $\pi_* \O_{\Y(k,1)}$ is torsion free and $\pi_{2*} Q$ is torsion. This cannot happen unless $\pi_{2*} Q = 0$ and we are done.
\end{proof}
\begin{Remark}
By considering the composition
$$\Y(\dots,1^{k_i+k_{i+1}},\dots) \rightarrow \Y(\dots,k_i,k_{i+1},\dots) \xrightarrow{\pi} \Y(\dots, k_i+k_{i+1},\dots)$$
it is not difficult to show using the result above and unique decomposition that for arbitrary $k_i,k_{i+1}$ we have
$$\pi_* \O_{\Y(\uk)} \cong \bigoplus_{\qbins{k_i+k_{i+1}}{k_i}} \O_{\Y(\dots,k_i+k_{i+1},\dots)} \{-k_ik_{i+1}\}.$$
\end{Remark}

\subsection{Checking all the properties}

\begin{Proposition}\label{prop:adjoints}
For $i \in I$ and $\ell \in \{0,1\}$ we have
\begin{align*}
(\sE_{i,\ell} \1_\uk)^L &\cong \1_\uk \sF_{i,-\ell} \{-\uk \cdot \alpha_i-1\} \\
(\sE_{i,\ell} \1_\uk)^R &\cong \1_\uk \sF_{i,-\ell} \{\uk \cdot \alpha_i+1\}.
\end{align*}
\end{Proposition}
\begin{proof}
We prove the first isomorphism (the second follows similarly). Consider the two maps
$$\Y(\uk) \xleftarrow{\pi_1} \Y_i^1(\uk) \xrightarrow{\pi_2} \Y(\uk+\alpha_i).$$
Notice that all three varieties have the same dimension $d=m \sum_i k_i$. Hence, as a sheaf in $D(\Y(\uk) \times \Y(\uk+r\alpha_i))$ we have
$$\O_{\Y_i^1(\uk)}^{-1} \cong \omega_{\Y_i^1(\uk)} \otimes \omega_{\Y(\uk) \times \Y(\uk+\alpha_i)}^{-1} [-d] \cong \omega_{\pi_1} \otimes \pi_2^* \omega_{\Y(\uk+\alpha_i)}^{-1}.$$
Thus
\begin{align*}
(\sE_{i,\ell} \1_\uk)^L
& \cong \omega_{\pi_1} \otimes \det(\sV'_i/\sV_i)^{\otimes \ell} \{-(k_i-1)\} \\
& \cong \O_{\Y_i^1(\uk)} \otimes \det(\sV'_i/\sV_i)^{\otimes \ell} \{k_i-1\} \cong \1_\uk \sF_{i,-\ell} \{k_i-k_{i+1}-1\}
\end{align*}
where the second isomorphism uses Lemma \ref{lem:dualizing}. The result follows since $\uk \cdot \alpha_i = -k_i + k_{i+1}$.
\end{proof}

\begin{Proposition}\label{prop:ij}
For $i \ne j \in I$ and $\ell,\ell' \in \{0,1\}$ we have
$$\sF_{j,-\ell'} * \sE_{i,\ell} \cong \sE_{i,\ell} * \sF_{j,-\ell'}.$$
\end{Proposition}
\begin{proof}
Suppose $j=i+1$ and $\ell=0=\ell'$ (the other cases are similar). Then
\begin{align*}
(\sF_j * \sE_i) \1_\uk
&\cong \O_{\Y^1_{i+1}(\uk+\alpha_i-\alpha_{i+1})} \{k_{i+2}-1\} * \O_{\Y^1_i(\uk)} \{k_i-1\} \\
&\cong \pi_{13*}(\O_{\pi_{12}^{-1} \Y^1_i(\uk)} \otimes \O_{\pi_{23}^{-1} \Y^1_{i+1}(\uk+\alpha_i-\alpha_{i+1})}) \{k_i+k_{i+2}-2\} \\
&\cong \pi_{13*}(\O_\Y) \{k_i+k_{i+2}-2\}
\end{align*}
where $\Y$ is the locus
$$\{\dots L_{i-1} \xrightarrow{k_{i-1}-1} L_i'=L_i'' \xrightarrow{1} L_i \xrightarrow{k_i} L_{i+1}=L_{i+1}' \xrightarrow{1} L_{i+1}'' \xrightarrow{k_{i+1}-1} L_{i+2} \dots\}$$
and $\pi_{13}$ forgets $L_i'$ and $L_{i+1}'$. This means that $\pi_{13}$ is an isomorphism of $\Y$ onto its image and hence
$$(\sF_j * \sE_i) \1_\uk \cong \O_{\Y(\dots,k_{i-1}-1,1,k_i,1,k_{i+1}-1,\dots)} \{k_i+k_{i+2}-2\}.$$
A similar argument shows that $(\sE_i * \sF_j) \1_\uk$ is the same kernel.
\end{proof}

\begin{Proposition}\label{prop:ef}
For $i \in I$ and $\ell \in \{0,1\}$ we have
\begin{align*}
(\sE_{i,\ell} * \sF_{i,-\ell}) \1_\uk \cong (\sF_{i,-\ell} * \sE_{i,\ell}) \1_{\uk} \bigoplus_{[-k_i+k_{i+1}]} \1_\uk \ \ & \text{ if } k_i \le k_{i+1} \\
(\sF_{i,-\ell} * \sE_{i,\ell}) \1_\uk \cong (\sE_{i,\ell} * \sF_{i,-\ell}) \1_{\uk} \bigoplus_{[k_i-k_{i+1}]} \1_\uk \ \ & \text{ if } k_i \ge k_{i+1}
\end{align*}
\end{Proposition}
\begin{proof}
We work out the case $\ell=0$ (the case $\ell=1$ is the same). The result will follow essentially formally from Propositions \ref{prop:main} and \ref{prop:ef}. First note that if $k_{i+1}=0$ then $(\sF_i * \sE_i) \1_\uk \cong \bigoplus_{[k_{i+1}]} \1_\uk$ follows immediately from Proposition \ref{prop:main}. Likewise for $(\sE_i * \sF_i) \1_\uk \cong \bigoplus_{[k_{i+1}]} \1_\uk$ if $k_i=0$.

We will now use this to imply the first assertion for arbtirary $k_i \le k_{i+1}$ (the second assertion for $k_i \ge k_{i+1}$ follows similarly). The proof is by induction on $k_i+k_{i+1}$. The base case of the induction is $k_i=k_{i+1}=1$ where we need to show that $(\sE_i * \sF_i) \1_\uk \cong (\sF_i * \sE_i) \1_\uk$. This follows easily since both sides correspond to $\pi^* \pi_*$ where $\pi$ forgets $L_i$.

Let $ \uk $ be an arbitrary sequence and let $\uk' = (\dots, k_i,0,k_{i+1},\dots)$ and note that we have a canonical identification $\Y(\uk)=\Y(\uk')$. Under this identification it follows (almost by definition) that $\sF_i \1_\uk$ can be identified with $(\sF_i * \sF_{i+1}) \1_{\uk'}$. Likewise $\1_\uk \sE_i, \sE_i \1_\uk$ and $\1_\uk \sF_i$ can be identified with $\1_{\uk'} (\sE_{i+1} * \sE_i), (\sE_{i+1} * \sE_i) \1_{\uk'}$ and $\1_{\uk'} (\sF_i * \sF_{i+1})$ respectively. Thus it remains to show that
\begin{equation}\label{eq:3}
(\sE_{i+1} * \sE_i) * (\sF_i * \sF_{i+1}) \1_{\uk'} \cong (\sF_i * \sF_{i+1}) * (\sE_{i+1} * \sE_i) \1_{\uk'} \bigoplus_{[-k_i+k_{i+1}]} \1_{\uk'}.
\end{equation}
To evaluate the left hand side of (\ref{eq:3}), we have
\begin{align*}
[\sE_{i+1} * (\sF_i * \sE_i) * \sF_{i+1}] \1_{\uk'}
&\cong \text{ LHS of (\ref{eq:3})} \bigoplus_{[k_i-1]} (\sE_{i+1} * \sF_{i+1}) \1_{\uk'} \\
&\cong \text{ LHS of (\ref{eq:3})} \bigoplus_{[k_{i+1}][k_i-1]} \1_{\uk'}
\end{align*}
where the first isomorphism follows by the induction hypothesis and the second from Proposition \ref{prop:main}. On the other hand, using Proposition \ref{prop:ij} we have
\begin{align*}
[\sE_{i+1} * \sF_i * \sE_i * \sF_{i+1}] \1_{\uk'}
&\cong [\sF_i * \sE_{i+1} * \sF_{i+1} * \sE_i] \1_{\uk'} \\
&\cong [\sF_i * \sF_{i+1} * \sE_{i+1} * \sE_i] \1_{\uk'} \bigoplus_{[k_{i+1}-1]} (\sF_i * \sE_i) \1_{\uk'} \\
&\cong (\sF_i * \sF_{i+1}) * (\sE_{i+1} * \sE_i) \1_{\uk'} \bigoplus_{[k_{i+1}-1][k_i]} \1_{\uk'}
\end{align*}
where the second isomorphism follows again by the induction hypothesis and the third by Proposition \ref{prop:main}. Since $[k_{i+1}-1][k_i] - [k_{i+1}][k_i-1] = [-k_i+k_{i+1}]$ it follows, using uniqueness of decomposition, that
$$\text{ LHS of (\ref{eq:3})} \cong (\sF_i * \sF_{i+1}) * (\sE_{i+1} * \sE_i) \1_{\uk'} \bigoplus_{[-k_i+k_{i+1}]} \1_{\uk'} = \text{ RHS of (\ref{eq:3})}.$$
\end{proof}

\begin{Proposition}\label{prop:triangle}
If $\la \alpha_i, \alpha_j \ra = -1$ then
\begin{equation}\label{eq:toshow}
\Cone(\sE_i * \sE_{j,1} \xrightarrow{\alpha} \sE_{j,1} * \sE_{i} \{1\}) \cong \Cone(\sE_{j} * \sE_{i,1} \xrightarrow{\beta} \sE_{i,1} * \sE_{j} \{1\})
\end{equation}
where $\alpha$ and $\beta$ are the unique (up to rescaling) nonzero maps.
\end{Proposition}
\begin{proof}
Suppose $j=i+1$ (the case $j=i-1$ is similar). Recall the uniqueness of $\alpha$ and $\beta$ follows as a formal consequence of the categorical action.

Since all the intersections of the convolutions have the expected dimension it is not difficult to check that
$$\sE_i * \sE_{j,1} \cong \O_{V_1} \otimes \det(\sV_j/\sV_j')^{-1} [j] \{k_i+k_{i+1}-2-j\}$$
where $V_1 = \{\dots \xrightarrow{k_{i-1}} L_{i-1} \xrightarrow{k_i-1} L_i' \xrightarrow{1} L_i \xrightarrow{k_{i+1}-1} L_{i+1}' \xrightarrow{1} L_{i+1} \xrightarrow{k_{i+2}} \dots\}$ and
$$\sE_{j,1} * \sE_{i} \cong \O_{V_2} \otimes \det(\sV_j/\sV_j')^{-1} [j] \{k_i+k_{i+1}-1-j\}$$
where $V_2 = \{\dots \xrightarrow{k_{i-1}} L_{i-1} \xrightarrow{k_i-1} L_i' \overset{1}{\underset{k_{i+1}}{\rightrightarrows}} \begin{matrix} L_i \\ L_{i+1}' \end{matrix} \overset{k_{i+1}}{\underset{1}{\rightrightarrows}} L_{i+1} \xrightarrow{k_{i+2}} \dots \}$. Now $V_2$ is the union of two components: one is $V_1$ and the other is the closure of the locus where the eigenvalues of $z|_{L_{i+1}/L'_{i+1}}$ and $z|_{L_i/L_i'}$ are the same (we call this $V_1'$).

Now, consider the standard exact triangle
$$\sI_{V_1 \cap V_1', V_1} \rightarrow \O_{V_1 \cup V_1'} \rightarrow \O_{V_1'}$$
where $\sI_{V_1 \cap V_1', V_1}$ is the ideal sheaf of $V_1 \cap V_1' \subset V_1$. In this case $V_1 \cap V_1' \subset V_1$ consists of the locus where the eigenvalues of $z|_{L_{i+1}'/L_{i+1}}$ and $z|_{L_i'/L_i}$ agree which is carved out by $x_{i+1}-x_i$ where $x_i,x_{i+1}: V_1 \rightarrow \A^1$ record these eigenvalues. It follows that $\sI_{V_1 \cap V_1', V_1} \cong \O_{V_1}\{-2\}$. Tensoring with $\det(\sV_j/\sV_j')^{-1}$ and the appropriate grading shift we end up with an exact triangle
$$\sE_i * \sE_{j,1} \rightarrow \sE_{j,1} * \sE_{i} \{1\} \rightarrow \O_{V_1'} \otimes \det(\sV_j/\sV_j')^{-1} [j] \{k_i+k_{i+1}-j\}.$$
It follows that
$$\Cone(\sE_i * \sE_{j,1} \xrightarrow{\alpha} \sE_{j,1} * \sE_{i} \{1\}) \cong \O_{V_1'} \otimes \det(\sV_j/\sV_j')^{-1} [j] \{k_i+k_{i+1}-j\}.$$
This gives us the left hand side of (\ref{eq:toshow}).

To simplify the right hand side of (\ref{eq:toshow}) first note that the same argument as above shows that
\begin{align*}
\sE_j * \sE_{i,1} &\cong \O_{V_2} \otimes \det(\sV_i/\sV_i')^{-1} [i] \{k_i+k_{i+1}-1-i\} \\
\sE_{i,1} * \sE_j &\cong \O_{V_1} \otimes \det(\sV_i/\sV_i')^{-1} [i] \{k_i+k_{i+1}-2-i\}.
\end{align*}
One then considers the exact triangle
$$\sI_{V_1 \cap V_1', V_1'} \rightarrow \O_{V_1 \cup V_1'} \rightarrow \O_{V_1}.$$
In this case $V_1 \cap V_1' \subset V_1'$ consists of the locus where $L_i \subset L_{i+1}'$. This is given as the vanishing locus of the following natural map of line bundles
$$\sV_i/\sV_i' \rightarrow \sV_{i+1}/\sV_{i+1}'.$$
It follows that
$$\sI_{V_1 \cap V_1', V_1'} \cong \O_{V_1'} \otimes \det(\sV_i/\sV_i') \otimes \det(\sV_{i+1}/\sV_{i+1}')^{-1}$$
which gives the exact triangle
$$\O_{V_1'} \otimes \det(\sV_{i+1}/\sV_{i+1}')^{-1} [i] \{k_i+k_{i+1}-1-i\} \rightarrow \sE_j * \sE_{i,1} \rightarrow \sE_{i,1} * \sE_j \{1\}$$
Then the same argument as before shows that
$$\Cone(\sE_j * \sE_{i,1} \xrightarrow{\beta} \sE_{i,1} * \sE_j \{1\}) \cong \O_{V_1'} \otimes \det(\sV_j/\sV_j')^{-1} [i+1] \{k_i+k_{i+1}-1-i\}.$$
The result follows since $j=i+1$.
\end{proof}

\begin{Proposition}\label{prop:theta}
For $i \in I$, $\ell \in \{0,1\}$ we have
$$\sE_{i,\ell} * \sE_{i,\ell} \cong \sE_{i,\ell}^{(2)} \{-1\} \oplus \sE_{i,\ell}^{(2)} \{1\}$$
where $\sE_{i,\ell}^{(2)} \cong \O_{\Y_i^2(\uk)} \otimes \det(\sV_i/\sV_i')^{\otimes -\ell} [2i \ell] \{2(k_i-2-i\ell)\}$. Moreover, for $\theta \in \spn \{\alpha_i: i \in I\}$ the map $I * \t * I \in \End^2(\sE_{i,\ell} * id * \sE_{i,\ell})$ induces a map between the summands $\sE_{i,\ell}^{(2)} \{1\}$ on either side which is
\begin{itemize}
\item nonzero if $\la \theta, \alpha_i \ra \ne 0$ and
\item zero if $\la \theta, \alpha_i \ra = 0$.
\end{itemize}
\end{Proposition}
\begin{proof}
We will consider the case $\ell=0$ (the case $\ell=1$ is the same). In this case
$$\sE_i * \sE_i \1_\uk\cong \pi_* \O_{\Y(\dots,k_i-2,1,1,k_{i+1},\dots)} \{2k_i-3\}$$
where $\pi: \Y(\dots,k_i-2,1,1,k_{i+1},\dots) \rightarrow \Y(\dots,k_i-2,2,k_{i+1},\dots)$ forgets the obvious flag. By Proposition \ref{prop:main} this is equal to
$$\bigoplus_{[2]} \O_{\Y(\dots,k_i-2,2,k_{i+1},\dots)} \{-1\} \{2k_i-3\} \cong \bigoplus_{[2]} \O_{\Y_i^2(\uk)} \{2(k_i-2)\} \cong \bigoplus_{[2]} \sE_i^{(2)}$$
which proves the first claim.

To prove the second claim consider the composition
$$\Y(\dots,k_i-2,1,1,k_{i+1},\dots) \rightarrow \Y(\dots,k_i-1,k_{i+1}+1,\dots) \xrightarrow{\ch} \A_{\uk+\alpha_i}$$
where $\A_{\uk+\alpha_i} = [\dots \times \A_{k_i-1} \times \A_{k_{i+1}+1} \times \dots]$. As before, we denote by $e_1^{(i)}$, for $i=1, \dots, n$, the degree two generators on the right hand side. For $j \ne i,i+1$ it is clear that
$$e_1^{(j)}: \pi_* \O_{\Y(\dots,k_i-2,1,1,k_{i+1},\dots)} \rightarrow \pi_* \O_{\Y(\dots,k_i-2,1,1,k_{i+1},\dots)} \{2\}$$
induces a map $\oplus_{[2]} \sE_i^{(2)} \rightarrow \oplus_{[2]} \sE_i^{(2)} \{2\}$ which is zero between the summands $\sE_i^{(2)} \{1\}$ on either side (because in this case $e_1^{(j)}$ is pulled back from $\Y(\dots,k_i-2,2,k_{i+1},\dots)$). Likewise the same is true of $(e_1^{(i)} + e_1^{(i+1)})$.

On the other hand, it is not hard to see following the proof in Proposition \ref{prop:main} that $e_1^{(i)}$ (or equivalently $e_1^{(i+1)}$) induces a nonzero map between the $\sE_i^{(2)}\{1\}$ summands. The second part of the Proposition now follows since under the map
$$\alpha_i \mapsto \mbox{ multiplication by } (e_1^{(i)} - e_1^{(i+1)})$$
it is elementary to check that $\la \theta, \alpha_i \ra = 0$ is equivalent to
$$\theta \mapsto \mbox{ multiplication by } \sum_i a_i e_1^{(i)}$$
where $a_i=a_{i+1}$.
\end{proof}
\begin{Remark}
The argument above also shows that $(\sE_{i,\ell})^{*r} \cong \oplus_{[r]!} \sE_{i,\ell}^{(r)}$ where
$$\sE_{i,\ell}^{(r)} \cong \O_{\Y_i^r(\uk)} \otimes \det(\sV_i/\sV_i')^{\otimes -\ell} [i \ell r] \{r(k_i-r-i \ell)\} \in D(\Y(\uk) \times \Y(\uk+r\alpha_i)).$$
Likewise one can show that $(\sF_{i,-\ell})^{*r} \cong \oplus_{[r]!} \sF_{i,-\ell}^{(r)}$ where
$$\sF_{i,-\ell}^{(r)} \cong \O_{\Y_i^r(\uk)} \otimes \det(\sV_i/\sV_i')^{\otimes \ell} [-ir \ell] \{r(k_{i+1}+i \ell)\} \in D(\Y(\uk+r\alpha_i) \times \Y(\uk)).$$
\end{Remark}

\begin{proof}[Proof of Theorem \ref{thm:main}]
Propositions \ref{prop:adjoints}, \ref{prop:ij}, \ref{prop:ef}, \ref{prop:triangle}, \ref{prop:theta} imply conditions (\ref{co:adj}), (\ref{co:EiFj}), (\ref{co:EF}), (\ref{co:triangle}), (\ref{co:theta}) from Section \ref{sec:catactions}.

To check condition (\ref{co:hom1}) notice that the $\C^\times$-equivariant map $\Y(\uk) \rightarrow \A_\uk$ is proper and the $\C^\times$-action on $\A_\uk$ contracts everything to zero. This implies that the space of maps between any two (coherent) sheaves is finite dimensional and also that $\Hom(\O_{\Y(\uk)}, \O_{\Y(\uk)} \{l\})$ is zero if $l < 0$ and one-dimensional if $l=0$. This completes the proof of Theorem \ref{thm:main}.
\end{proof}

\subsection{Some remarks on shifts}\label{sec:shift}

It is a bit strange that the definitions of $\sE_{i,1}$ and $\sF_{i,-1}$ at the beginning of this section involve the shifts $[i]\{-i\}$. These shifts are necessary in Proposition \ref{prop:triangle}. If one removes these shifts, let us call these new kernels $\sE'_{i,1}$ and $\sF'_{i,-1}$, then we find that for $j=i+1$ we have
$$\Cone(\sE_i * \sE'_{j,1} \xrightarrow{\alpha} \sE'_{j,1} * \sE_i \{1\}) \cong \Cone(\sE_{j} * \sE'_{i,1} \xrightarrow{\beta} \sE'_{i,1} * \sE_{j} \{1\}) [-1]\{1\}.$$
We can work with this definition if we just change condition (\ref{co:triangle}) in Section \ref{sec:catactions} to say that $\Cone(\alpha) \cong \Cone(\beta) [-1]\{1\}$ if $j-i=1$. This is less natural in some cases (for example, for the action constructed in \cite{CL}) but, as we now explain, it makes things cleaner for the action in this paper (as well as the one in  \cite{CaKo}).

First, this changed definition has an effect on the affine braid group discussion from Section \ref{sec:braids}. Using the shifted complexes $\T'_i$ from the end of that section it is straight-forward to check that Lemma \ref{lem:conjugation} now reads
$$\T'_i \E'_{j,1} (\T'_i)^{-1} \cong \T'_j \E'_{i,1} (\T'_j)^{-1}$$
assuming $\la i,j \ra = -1$. So if we use $\phi' = \T'_{i,1} \T'_i$ then the argument from Section \ref{sec:braids} goes through to give an affine braid group action generated by $\T'_i$ and $\phi'_i$. The advantage of this braid group action is that, as we will see in Section \ref{sec:equiv}, it can be described more explicitly without the need of shifts. For example, the $\phi'$ turn out to be tensoring with some simple line bundle, whereas the $\phi$ are given by tensoring with a somewhat complicated shift of this line bundle.

\section{K-theory and symmetric Howe duality}\label{sec:Ktheory}

In \cite{CKL1}, we constructed a categorical $\sl_2$ action on $\oplus_{(k,l)} D(Y(k,l))$. This action generalizes to give an $\sl_n$ action on $\oplus_\uk D(Y(\uk))$, as explained in \cite{C1}. These varieties were recalled in Section \ref{sec:Ys}, but for the moment we just note that they are iterated Grassmannian bundles. Hence one has the identification
\begin{equation}\label{eq:K1}
K_{\Cx}(Y(\uk)) \cong \Lambda^{k_1}(\C^m) \otimes \dots \otimes \Lambda^{k_n}(\C^m)
\end{equation}
where $K_{\Cx}(X)$ denotes the Grothendieck group of equivariant coherent sheaves on $X$ tensored over $\Z[q,q^{-1}]$ (the $\Cx$-equivariant Grothendieck group of a point) with $\C(q)$. Note that in all cases that we consider the algebraic and topological K-theory will be the same. This action categorifies the natural $U_q(\sl_n)$ action on $\Lambda^N(\C^n \otimes \C^m)$ where $N = \sum_i k_i$. Using skew Howe duality one can use this action to obtain the R-matrix for $U_q(\sl_m)$ and subsequently the braid group action on the right hand side of (\ref{eq:K1}).

On the other hand there is an analogous story involving symmetric powers. In this case the big space is $\Sym^N(\C^n \otimes \C^m)$ and symmetric Howe duality can be used to construct the braid group action on $\Sym^{k_1}(\C^m) \otimes \dots \otimes \Sym^{k_n}(\C^m)$ (see \cite{TL}, Theorem 6.5). A categorified version of this action is given by the categorical $\sl_n$ action on $\oplus_\uk D(\Y(\uk))$. Lemma \ref{lem:Ktheory} below explains that the Grothendieck groups of $\Y(\uk)$ do indeed have the correct size.

\begin{Lemma}\label{lem:Ktheory}
We have a natural identification
\begin{equation}\label{eq:K2}
K_\Cx(\Y(\uk)) \cong \Sym^{k_1}(\C^m) \otimes \dots \otimes \Sym^{k_n}(\C^m).
\end{equation}
\end{Lemma}
\begin{proof}
Since $\Y(\uk)$ is an iterated product of $\Y(k)$s it suffices to show that $K_\Cx(\Y(k)) \cong \Sym^k (\C^m)$. Here are two ways to do this.

First, note that $K_\Cx(\Y(k)) \cong H^*_{\Cx}(\Y(k))$. On the other hand, we have the map $\ch: \Y(k) \rightarrow \A_k$ and a $\C^\times$ action which contracts everything to $\ch^{-1}(0)$. So $H^*_\Cx(\Y(k)) \cong H^*_\Cx(\ch^{-1}(0))$. On the other hand, $\ch^{-1}(0)$ is isomorphic to the Schubert variety $\overline{\GrBD^{k \omega_1}}$ and by geometric Satake $IH(\overline{\GrBD^{k \omega_1}}) \cong V_{k \omega_1} \cong \Sym^k(\C^m)$. Finally, because each weight space of this representation is one dimensional it actually follows that $IH(\overline{\GrBD^{k \omega_1})}) \cong H^*(\overline{\GrBD^{k \omega_1}}))$ and hence $K_\Cx(\Y(k)) \cong \Sym^k(\C^m)$.

Alternatively, we can use localization. Note that $GL_m(\C) \times \C^\times$ acts on $\Y(k)$. Choose the $\C^\times \subset GL_m(\C) \times \C^\times$ which looks like
$$\diag(1,t,t^2,\dots,t^m) \times (t).$$
Using the notation from Section \ref{sec:geometry} we identify $\Y(k)$ with $\X(k)$. Then the $\C^\times$ fixed points are lattices $L \subset L_0$ where $L$ is generated by
$$z^{\mu_1}e_1, z^{\mu_2}e_2, \dots, z^{\mu_m}e_m$$
where $\mu_i \in \N$ satisfy $\sum_i \mu_i = k$. It is easy to see that the number of such points is $\binom{m+k-1}{k} = \dim \Sym^k(\C^m)$.
\end{proof}

Since a representation of $ U_q(\sl_n) $ is determined by the dimension of its weight spaces, we immediately deduce the following corollary.
\begin{Corollary} \label{co:onKgroup}
We have an isomorphism of $ U_q(\sl_n) $ representations
$$ \bigoplus_{\uk} K_{\Cx}(\Y(\uk)) \cong \Sym^N(\C^n \otimes \C^m) $$
where the direct sum ranges over those sequences $ \uk \in \N^n $ with $\sum_i k_i = N$.
\end{Corollary}

\section{Braid group actions} \label{sec:equiv}

We will now try to describe more explicitly the kernels corresponding to $\T_i$ and $\phi_i$ which, by Theorem \ref{thm:affinebraid}, give us an affine braid group action on $\oplus_\uk D(\Y(\uk))$. In fact, it turns out it is nicer to describe the shifted kernels of $\T'_i$ and $\phi'_i$ (c.f. Section \ref{sec:shift}).

\subsection{The case of $ \Y(1^n) $}
We begin with the braid group action on $D(\Y(1^n))$. We consider this case first, since the action here is simpler to study, and also because this weight space is the one which links skew and symmetric Howe duality (see Section \ref{sec:skewsym}).

Inside $\Y(1^n)$ consider the open locus $\Y^o(1^n)$ where the eigenvalues of $z \in \End(L_n/L_0)$ are all distinct. For $i \in I$ define $\Z_i^o(1^n,1^n) \subset \Y^o(1^n) \times \Y^o(1^n)$ as the locus $\{(L_\bullet, L'_\bullet)\}$ where $L_n=L_n'$ and if $z|_{L_\bullet/L_0}$ has eigenvalues $(\l_1, \dots, \l_n)$ then $z|_{L'_\bullet/L_0}$ has eigenvalues $(\l_1, \dots, \l_{i+1},\l_i, \dots, \l_n)$ (i.e. we switched eigenvalues $\l_i$ and $\l_{i+1}$). Notice that $\Z_i^o(1^n,1^n)$ is just the graph of an automorphism of $\Y^o(1^n)$. We then denote
$$\Z_i(1^n,1^n) := \overline{\Z_i^o(1^n,1^n)} \subset \Y(1^n) \times \Y(1^n).$$

\begin{Proposition}\label{prop:braid}
The kernels
$$\O_{\Z_i(1^n,1^n)} \otimes (\sV_i/\sV_{i-1}) \otimes (\sV'_{i+1}/\sV'_i)^{-1} \{1\} \text{ and } \O_{\Z_i(1^n,1^n)} \{1\} $$
induce the braid element $\T'_i$ and its inverse acting on $D(\Y(1^n))$. Similarly, the kernels
$$\O_{\Z_i(1^n,1^n)} \otimes (\sV_{i+1}'/\sV'_i)^{-1} \otimes (\sV_i'/\sV'_{i-1}) \{1\} \text{ and } \O_{\Z_i(1^n,1^n)} \otimes (\sV_i/\sV_{i-1})^{-1} \otimes (\sV'_i/\sV'_{i-1}) \{1\} $$
induce the braid element $\T'_{i,1}$ and its inverse acting on $D(\Y(1^n))$.
\end{Proposition}
\begin{proof}
The kernel for $\T_i$ in this case is
$$\sT_i := \Cone(\sE_i * \sF_i \{-1\} \rightarrow \O_{\Delta}).$$
Since everything intersects in the expected dimension it is straight-forward to check that
$$\sE_i * \sF_i \cong \O_{\W_i} \{1\} \in D(\Y(1^n) \times \Y(1^n))$$
where $\W_i = \{L_\bullet, L'_\bullet: L_j = L_j' \text{ for } j \ne i \}$. Now $\W_i$ is the union of two components: the diagonal $\Delta$ (where $L_i=L_i'$) and the locus $\Z_i(1^n,1^n)$ (the closure of the locus where $L_i \ne L_i'$). The intersection $\Delta \cap \Z_i(1^n,1^n) \subset \Z_i(1^n,1^n)$ is the locus where $L_i=L_i'$. In other words, this is the locus where the map $\sV_i/\sV_{i-1} \rightarrow \sV'_{i+1}/\sV_i'$ vanishes. This gives
$$\O_{\Z_i(1^n,1^n)}(-\Delta \cap \Z_i(1^n,1^n)) \cong \O_{\Z_i(1^n,1^n)} \otimes (\sV_i/\sV_{i-1}) \otimes (\sV'_{i+1}/\sV'_i)^{-1}.$$
Thus, from the standard exact triangle involving $\Delta,\Z_i(1^n,1^n)$ and $\W_i$ we get the exact triangle
$$\O_{\Z_i(1^n,1^n)} \otimes (\sV_i/\sV_{i-1}) \otimes (\sV'_{i+1}/\sV'_i)^{-1} \rightarrow \O_{\W_i} \rightarrow \O_\Delta.$$
Comparing with the exact triangle
$$\sT_i [-1] \rightarrow \sE_i * \sF_i \{-1\} \cong \O_{\W_i} \rightarrow \O_{\Delta}$$
we get that $\sT_i [-1] \cong \O_{\Z_i(1^n,1^n)} \otimes (\sV_i/\sV_{i-1}) \otimes (\sV'_{i+1}/\sV'_i)^{-1}$. The result follows since $\sT_i' = \sT_i [-1]\{1\}$.

A similar argument shows that the kernel for $\T_i^{-1}$ is $\O_{\Z_i(1^n,1^n)} \{1\}$. More precisely, one uses the exact triangle
$$\O_{\Delta}(-\Delta \cap \Z_i(1^n,1^n)) \rightarrow \O_{\W_i} \rightarrow \O_{\Z_i(1^n,1^n)}$$
together with the fact that $\O_{\Delta}(-\Delta \cap \Z_i(1^n,1^n)) \cong \O_{\Delta} \{-2\}$. To see this last fact consider the map
$$\ch: \Y(1^n) \rightarrow \A^n = \Spec \C[x_1, \dots, x_n]$$
and note that $\Delta \cap \Z_i(1^n,1^n) \subset \Delta$ is the locus where $x_i = x_{i+1}$.

The computations of the kernels for $\T_{i,1}$ and its inverse are similar.
\end{proof}

\subsection{The general case (conjectural)}

We can try to generalize Proposition \ref{prop:braid} to $\T'_w \1_\uk$ where $w \in S_n$ and $\T'_w$ is the composition of $\T'_i$s corresponding to $w$ (here we use the usual inclusion of the symmetric group into the braid group).

To do this recall the map $ \ch : \Y(\uk) \rightarrow \A_{\uk}$ and denote again by $\Y^o(\uk)$ the open locus where the eigenvalues of $z|_{L_n/L_0}$ are all distinct. For $w \in S_n$ the natural action $\A_\uk \xrightarrow{\sim} \A_{w \cdot \uk}$ lifts to a commutative diagram
$$\xymatrix{
\Y^o(\uk) \ar[rr]^{w \cdot } \ar[d]^\ch & & \Y^o(w \cdot \uk) \ar[d]^\ch \\
\A_\uk \ar[rr]^{w \cdot } & & \A_{w \cdot \uk}
}$$
The lift is uniquely determined by the condition that if $ w(L_\bullet) = L'_\bullet $, then $L_n = L_n'$ (together with the commutativity of the diagram above). 

We let $\Z_w^o(\uk, w \cdot \uk) \subset \Y^o(\uk) \times \Y^o(w \cdot \uk) $ denote the graph of the action of $ w \in S_n $ and we denote
$$ \Z_w(\uk, w \cdot \uk) := \overline{\Z_w^o(\uk, w \cdot \uk)} \subset \Y(\uk) \times \Y(w \cdot \uk).$$
The following conjecture generalizes a result of Bezrukavnikov-Riche \cite{BR}.

\begin{Conjecture}
For any $ w \in S_n$ the kernel $\O_{\Z_w(\uk, w \cdot \uk)} \{\ell_\uk(w)\}$ induces the inverse of the braid element $\T'_w \1_\uk$.
\end{Conjecture}

Here $\ell_\uk(w)$ denotes the weighted length of $w$ where an involution exchanging $k$ and $k'$ has weight $kk'$. In particular, $\ell_{\uk}(s_i) = k_i k_{i+1}$ while $\ell_{(1^n)}(w)$ is the usual length of $w$.

\subsection{The lattice action}

Finally, we would also like to have an explicit description of the $\phi_i$. Notice that these generated an action of a lattice. Let us denote by $\sL_i$ the line bundle $\det(\sV_i/\sV_{i-1})$.

\begin{Corollary}\label{cor:braid}
The functor $\phi'_i \1_\uk = \T'_{i,1} \T'_i \1_\uk$ is given by tensoring with the line bundle $\sL_{i+1}^{-1} \otimes \sL_i$.
\end{Corollary}
\begin{Remark}
If we use our original $(L\gl_n,\t)$ action (with the shifts) then $\phi_i \1_\uk$ is given by tensoring with $\sL_{i+1}^{-1} \otimes \sL_i [s]\{-s\}$ where $s = i \uk \cdot \alpha_i + k_i+k_{i+1}$.
\end{Remark}
\begin{proof}
If $\uk = 1^n$ then this follows immediately from Proposition \ref{prop:braid} by noting that $\T'_{i,1}$ is equal to the inverse of $\T'_i$ up to tensoring by some line bundles on one side. We now recursively prove this for the other weights.

More precisely, we will be using Lemma \ref{lem:phiE}, which also holds if you replace $\phi_i$ with $\phi'_i$ and $\E_{i,1}$ with $\E'_{i,1}$. Now consider $\E_j \1_\uk$ with $\uk \cdot \alpha_j < 0$ and suppose we know that, for all $i \in I$, $\phi'_i \1_{\uk + \alpha_j}$ is isomorphic to the functor $\gamma_i$ induced by tensoring with $\sL_{i+1}^{-1} \otimes \sL_i$. As in the definition of $\phi_\alpha$, we will denote
$$\gamma_\alpha := \prod_{i \in I} \gamma_i^{a_i}$$
whenever $\alpha = \sum_{i \in I} a_i \alpha_i$.

If $\la \alpha, \alpha_j \ra = 0$ then we get
\begin{align*}
\bigoplus_{- \uk \cdot \alpha_j} \gamma_\alpha^{-1} \phi'_\alpha \1_\uk \oplus \gamma_\alpha^{-1} \E_j \F_j \phi'_\alpha \1_\uk
&\cong \gamma_\alpha^{-1} \F_j \E_j \phi'_\alpha \1_\uk \\
&\cong \F_j \gamma_\alpha^{-1} \phi'_\alpha \E_j \1_\uk \\
&\cong \F_j \E_j \1_\uk \\
&\cong \bigoplus_{- \uk \cdot \alpha_j} \1_\uk \oplus \E_j \F_j \1_\uk
\end{align*}
where the second line follows using Lemma \ref{lem:phiE} and the fact that $\gamma_\alpha$ commutes with $\F_j$ (this is not hard to check), the third line uses the fact that $\gamma_\alpha \cong \phi'_\alpha$ on $D(\Y(\uk+\alpha_j))$. Unique decomposition then implies that $\gamma_\alpha^{-1} \phi'_\alpha \1_\uk \cong \1_\uk$.

On the other hand, if $\la i,j \ra = -1$ then we repeat the argument above with $\alpha_i$ instead of $\alpha$. Using Lemma \ref{lem:phiE} together with the fact that we have $\gamma_i \E_{j,1} \cong \E_j \gamma_i$ (essentially from definition) we find that $\gamma_i^{-1} \phi'_i \1_\uk \cong \1_\uk$. This completes the proof.
\end{proof}

\section{Relation to categorical geometric skew Howe duality}\label{sec:skewhowe}

We would like to explain three relationships between the geometric categorical skew and symmetric Howe duality pictures. The first involves skew-sym Howe duality as studied in \cite{RT,TVW}, the second (and related to the first) is a generalization of the main result from \cite{CK4} and the third involves linear Koszul duality.

\subsection{Skew-Sym Howe duality}\label{sec:skewsym}

In \cite{CKM} we used skew Howe duality to give a generators and relations description of the category of $U_q(\sl_m)$-representations generated by exterior powers of the standard representation (i.e. where the objects are $\Lambda^{k_1}(\C^m) \otimes \dots \otimes \Lambda^{k_n}(\C^m)$). We call this the skew side.

On the other hand, one can consider symmetric products instead of exterior products and study the category where the objects are of the form $\Sym^{k_1}(\C^m) \otimes \dots \otimes \Sym^{k_n}(\C^m)$. We call this the sym side. One can study this category using symmetric Howe duality and try to give a generators and relations description. This was done more recently in \cite{RT,TVW}.

In both cases it is relatively easy to give a set of generators for the space of morphisms. The more difficult step is to deal with the relations these generators satisfy. This is more difficult on the sym side because symmetric Howe duality does not provide enough relations. The recent insight of \cite{RT,TVW} is that one should study both sides together. It turns out that the missing relations can be encoded in one relation (\cite{TVW} relation 10) which relates the spaces $\End((\C^m)^{\otimes n})$ on both sides (note that $(\C^m)^{\otimes n}$ is the only object which belongs to both sides).

The 2-category $\tK^{n,N}_{\GrBD,m}$ in the current paper gives a geometric categorification of the sym side. In other words, using Lemma \ref{lem:Ktheory}, we know that taking Grothendieck groups we recover the tensor products of various $\Sym^k(\C^m)$. On the other hand, in prior work starting with \cite{CKL1}, we gave a geometric categorification of the skew side consisting of a 2-category $\K^{n,N}_{\GrBD,m}$, which is defined using $Y$s instead of $\Y$s (see Section \ref{sec:Ys}).

We now explain the geometric version of the extra relation from \cite{TVW} which ``glues together'' $\tK^{n,N}_{\GrBD,m}$ and $\K^{n,N}_{\GrBD,m}$. To do this we restrict to the case $N=n$. Consider the map $\ch: \Y(1^n) \rightarrow \A^n$. The fiber over $0 \in \A^n$ is $Y(1^n)$ and we denote $\iota: Y(1^n) \rightarrow \Y(1^n)$ its inclusion.

On $\Y(1^n) \times \Y(1^n)$ we have kernel $\sT_i$ inducing $\T_i$. On the other hand, we also have a categorical $\sl_n$ action on $\K^{n,N}_{{\sf Gr},m}$ which induces a braid group action on $Y(1^n)$. Abusing notation, we again denote the corresponding kernel on $Y(1^n) \times Y(1^n)$ by $\sT_i$.

\begin{Proposition}\label{prop:iota}
Consider the inclusions
$$Y(1^n) \times Y(1^n) \xrightarrow{id \times \iota} Y(1^n) \times \Y(1^n) \xrightarrow{\iota \times id} \Y(1^n) \times \Y(1^n).$$
Then we have $(\iota \times id)^* \sT_i \cong (id \times \iota)_* \sT_i [1]\{-2\}$.
\end{Proposition}
\begin{Remark}\label{rem:Ts}
Note that if we use $\sT'_i$ instead of $\sT_i$ then the relation above simplifies to give $(\iota \times id)^* \sT'_i \cong (id \times \iota)_* \sT'_i$. This is because we shift both sides by $[-1]\la1\ra$ but on the $\tK^{n,N}_{\GrBD,m}$ side $\la 1 \ra = \{1\}$ while on the $\K^{n,N}_{\GrBD,m}$ side one has $\la 1 \ra = [1]\{-1\}$.
\end{Remark}
\begin{proof}
Let us first find an explicit expression for $\sT_i$ on $D(Y(1^n) \times Y(1^n))$. This is a standard calculation which has been done on several occasions going back to \cite{CK2} but we sketch it here again for completeness.

Note that the particular choice of line bundles we are using to define the functors $\E_i$ and $\F_i$ are the ones from \cite{CKL1} (the one in \cite{C1} differs by conjugation with some line bundle and shift). In particular, $\E_i \1_\uk$ and $\1_\uk \F_i$ are induced by
\begin{align*}
\O_{W_i^1(\uk)} \otimes \det(\sV'_i/\sV_{i-1}) \otimes \det(\sV_{i+1}/\sV_i)^{-1} \{k_i-1\} \in D(Y(\uk) \times Y(\uk+\alpha_i)) & \ \text{ and } \\
\O_{W_i^1(\uk)} \otimes \det(\sV'_i/\sV_i)^{k_{i+1}-k_i+1} \{k_{i+1}\} \in D(Y(\uk+\alpha_i) \times Y(\uk)) &
\end{align*}
where $W_i^1(\uk) \subset Y(\uk) \times Y(\uk+\alpha_i)$ is the usual natural correspondence.

Now, the analogue of $\Z_i(1^n,1^n)$ is
$$Z_i(1^n,1^n) = \{(L_\bullet,L'_\bullet): L_j = L_j' \text{ for } j \ne i \}.$$
This variety has two components: the locus $\Delta$ where $L_i=L'_i$ and the locus $W_i$ where $zL_{i+1} \subset L_{i-1}$. We also get the standard exact sequence
$$\O_\Delta(-W_i) \rightarrow \O_{Z_i(1^n,1^n)} \rightarrow \O_{W_i}.$$
The divisor $W_i \cap \Delta \subset \Delta$ is carved out by the vanishing of $z: \sV_{i+1}/\sV_i \rightarrow \sV_i/\sV_{i-1} \{2\}$. Thus $\O_\Delta(-W_i) \cong \O_\Delta \otimes \sL \{-2\}$ where $\sL = (\sV'_{i+1}/\sV'_i) \otimes (\sV_i/\sV_{i-1})^{-1}$. It follows that we have the exact triangle
$$\O_{W_i} \otimes \sL^{-1} [-1] \{2\} \rightarrow \O_\Delta \rightarrow \O_{Z_i(1^n,1^n)} \otimes \sL^{-1} \{2\}.$$
On the other hand, it is straight-forward to check that
$$\sE_i * \sF_i \cong \O_{W_i} \otimes \sL^{-1} \{1\}$$
and so we have the exact triangle
$$ \O_{W_i} \otimes \sL^{-1} [-1] \{2\} \cong \sE_i * \sF_i \la -1 \ra \rightarrow \O_\Delta \rightarrow \sT_i.$$
Comparing these two exact triangles we get $\sT_i \cong \O_{Z_i(1^n,1^n)} \otimes \sL^{-1} \{2\}$.

On the other hand, notice that the restriction of $\Z_i(1^n,1^n)$ from $\Y(1^n) \times \Y(1^n)$ to $Y(1^n) \times \Y(1^n)$ is of the expected dimension and is equal to the image of $Z_i(1^n,1^n)$ under $(id \times \iota)$. This means that $(\iota \times id)^* \O_{\Z_i(1^n,1^n)} \cong (id \times \iota)_* \O_{Z_i(1^n,1^n)}$. Keeping track of line bundles and shifts we have
\begin{align*}
(\iota \times id)^* \sT_i
& \cong (id \times \iota)_* (\O_{\Z_i(1^n,1^n)}) \otimes (\sV_i/\sV_{i-1}) \otimes (\sV_{i+1}'/\sV'_i)^{-1} [1] \\
& \cong (id \times \iota)_* (\O_{Z_i(1^n,1^n)} \otimes \sL^{-1} [1]) \cong (id \times \iota)_* \sT_i [1] \{-2\}.
\end{align*}
\end{proof}

Proposition \ref{prop:iota} is a categorical analogue of the ``dumbbell'' relation from \cite{RT,TVW}. Let us explain this statement in more detail.

First note that Proposition \ref{prop:iota} implies that $\T_i \circ \iota_* \cong \iota_* \circ \T_i [1]\{-2\}$. This is a standard calculation (see for instance section 5.3 of \cite{CKL3} where we discuss compatible kernels). Taking adjoints and the analogous result for $\T'_i$ also implies that $\iota^* \circ \T_i \cong \T_i \circ \iota^* [1]\{-2\}$.

On the other hand, $\iota^*: D(\Y(1^n)) \rightarrow D(Y(1^n))$ is an isomorphism at the level of K-theory. So we can use this map to identify $K(\Y(1^n))$ and $K(Y(1^n))$.

Now, $[\sT_i] \in K(\Y(1^n) \times \Y(1^n))$ is equal to $\id - q^{-1} [\sE_i] * [\sF_i]$ whereas $[\sT_i] \in K(Y(1^n) \times Y(1^n))$ is equal to $\id + q [\sE'_i] * [\sF'_i]$ where we use the primes to remember that we are on the $Y(1^n)$ side (rather than the $\Y(1^n)$). The difference in the factors of $q$ is that $\la 1 \ra = \{1\}$ in the first case while $\la 1 \ra = [1]\{-1\}$ in the second.

Hence, once we identify $K(\Y(1^n))$ with $K(Y(1^n))$, the relation $\iota^* \circ \T_i \cong \T_i \circ \iota^* [1]\{-2\}$ implies that
$$\id - q^{-1} [\sE_i] * [\sF_i] = -q^{-2} (\id + q [\sE'_i] * [\sF'_i])$$
at the level of K-theory. Rearranging this becomes
\begin{equation}\label{eq:dumbbell}
[2] \cdot \id = [\sE_i] * [\sF_i] - [\sE'_i] * [\sF'_i].
\end{equation}
This is the dumbbell relation (10) from \cite{TVW} with one sign difference.

To (partially) explain this sign difference we note that this skew-sym story should be equipped with a natural involution that exchanges both sides. This means that the relations should be invariant under $\sE_i \mapsto \sE_i'$, $\sF_i \mapsto \sF'_i$ and $\la 1 \ra \rightarrow \la 1 \ra$. Now, if at the level of Grothendieck groups you record the shift $\la 1 \ra$ by $q$ on both sides then it makes sense to have the dumbbell relation
$$[2] \cdot \id = [\sE_i] * [\sF_i] + [\sE'_i] * [\sF'_i].$$
This is what happens in \cite{TVW}. On the other hand, in our case $\la 1 \ra = \{1\}$ is recorded by $q$ on one side whereas $\la 1 \ra = [1]\{-1\}$ is recorded by $-q^{-1}$ on the other (this is just because we want the internal grading shift $\{1\}$ to always be recorded by $q$). This means that the dumbbell relation should be invariant under $q \mapsto -q^{-1}$. Notice that (\ref{eq:dumbbell}) is indeed invariant under the involution $\sE_i \mapsto \sE_i'$, $\sF_i \mapsto \sF'_i$ and $q \mapsto -q^{-1}$.

\subsection{K-theoretic geometric Satake}

In \cite{CK4} we described a K-theoretic version of the geometric Satake equivalence. This consisted of a pair of equivalences
$$KConv^{SL_n \times \C^\times}({\sf Gr}) \xleftarrow{\sim} \mathcal{AS}p_n \xrightarrow{\sim} \O_q^{min}(\frac{SL_n}{SL_n})\mod.$$
Here $KConv^{SL_n \times \C^\times}({\sf Gr})$ is the convolution category defined from the affine Grassmannian ${\sf Gr}$ of $SL_m$, $\mathcal{AS}p_n$ is an annular version of the spider category studied in \cite{CKM} and $\O_q^{min}(\frac{SL_n}{SL_n})\mod$ is a subcategory of all $SL_n$-equivariant $\O_q(SL_n)$-modules (where we are using the adjoint action) consisting of objects of the form
$$\O_q(SL_n) \otimes \Lambda^{k_1}(\C^m) \otimes \dots \otimes \Lambda^{k_n}(\C^m)$$
(the ``min'' stands for minuscule).

The category $KConv^{SL_n \times \C^\times}({\sf Gr})$ is constructed using the spaces $Y(\uk)$. One can similarly define a category using the spaces $\Y(\uk)$ or combine these categories into a larger category which for lack of better notation we denote $KConv^{SL_n \times \C^\times}_{skew-sym}({\sf Gr})$.

One can also define an extended category $\mathcal{AS}p'_n$. This is the annularization of the category from \cite{RT,TVW}. Finally, one can consider $\O_q^{skew-sym}(\frac{SL_n}{SL_n})\mod$ which is the same category as before but where the objects include both skew and sym powers of $\C^m$.

Then we expect the following equivalences of categories
$$KConv^{SL_n \times \C^\times}_{skew-sym}({\sf Gr}) \xleftarrow{\sim} \mathcal{AS}p'_n \xrightarrow{\sim} \O_q^{skew-sym}(\frac{SL_n}{SL_n})\mod.$$
We hope to explore this story in more detail in future work.

\subsection{Linear Koszul duality}
In this section, we consider those $ \Y(\uk) $ with $m=n=N$ where $N=k_1+\cdots+k_n$.

The Mirkovi\'c-Vybornov isomorphism, Theorem \ref{th:MViso}, restricts to an isomorphism $ \X(n)_0 = \X(n)_1 \cong M_n $, where $ M_n $ denotes the set of all $ n \times n $ matrices.  Let $ \Y(\uk)_0 $ denote the preimage of $ \X(n)_0 $ under the map $ \Y(\uk) \rightarrow \Y(n) \cong \X(n) $.  Then it is easy to see that we have an isomorphism $ \Y(\uk)_0 \cong M_n(\uk) $ where
$$M_n(\uk) := \{ (A, 0= V_0 \subseteq V_1 \subseteq \cdots \subseteq V_n= \C^n) : A \in M_n, AV_i \subset V_i, \dim (V_i/V_{i-1}) = k_i \}.$$
We can define $\tK_{Fl,n}$ to be the triangulated 2-category whose objects are indexed by $\uk$ (with $k_1+\dots+k_n = n$), the 1-morphisms are kernels inside $D(M_n(\uk) \times M_n(\uk'))$ and 2-morphisms are maps between kernels. The following is an immediate corollary of Theorem \ref{thm:main}.

\begin{Corollary}\label{cor:action1}
The $(L\gl_n, \theta)$ action on $\tK^{n,n}_{\GrBD,n}$ from Theorem \ref{thm:main} restricts to give an $(L\gl_n,\theta)$ action on $\tK_{Fl,n}$.
\end{Corollary}

This action is quite natural. For example, the kernel for $\E_i \1_\uk$ is given by the structure sheaf of the subvariety
$$ M_n(k_1, \dots, k_i -1, 1, k_{i+1}, \dots, k_n) \subset M_n(\uk) \times M_n(\uk + \alpha_i).$$
Considering the dimensions of the weight spaces (similar to Corollary \ref{co:onKgroup}) we find that the $(L\gl_n,\t)$ action on $ \tK_{Fl,n} $ categorifies the $U_q(L\gl_n)$ representation $ (\C^n)^{\otimes n} $.

We can further consider the  DG-schemes $ M_n(\uk) \times^L_{M_n} \{0\} $ and DG-coherent sheaves on these varieties and their products.  These DG schemes come with an action of $ \Cx $, left over from the scaling action on $ M_n $.  We define a 2-category $ \tK_{Fl,n}^{DG} $ whose objects are sequences $ \uk $ as before and whose morphisms are DG-coherent sheaves on products $ (M_n(\uk) \times^L_{M_n} \{0\})  \times (M_n(\uk') \times^L_{M_n} \{0\}) $. The action from Corollary \ref{cor:action1} gives us an $(L\gl_n,\theta)$ action on $ \tK_{Fl,n}^{DG} $.

Our interest in these DG-schemes is motivated by the linear Koszul duality of Mirkovi\'c-Riche (see \cite{Ri}, Theorem 2.3.10) which specializes to the following statement in our case.

\begin{Theorem} \label{th:KoszulDuality}
There is an equivalence of categories
$$ D( M_n(\uk) \times^L_{M_n} \{0\} ) \cong D(T^* Fl(\uk)). $$
\end{Theorem}

In this equivalence we consider $ \Cx $ equivariant objects on both sides. However, the shifts $ \{1\}$ of equivariant structure do not match under this equivalence.  More precisely, we have that $ \{ 1\} $ on the left hand side is sent to $\{1\}[-1]$ on the right hand side.

The space $T^* Fl(\uk)$ denotes the cotangent bundle of $Fl(\uk)$. It can be described as
$$ \{ (A, 0= V_0 \subseteq V_1 \subseteq \cdots \subseteq V_n= \C^n) : A \in M_n, AV_i \subset V_{i-1}, \dim (V_i/V_{i-1}) = k_i \}.$$
Note that $M_n(\uk)$ is also a vector bundle over $ Fl(\uk) $.  In fact, it is the perpendicular vector bundle to $T^* Fl(\uk)$ with respect to the usual bilinear form on matrices --- this is the source for the above Koszul duality equivalence.

In \cite{CK3} we constructed a categorical $\gl_n$ action on $\oplus_\uk D(T^*Fl(\uk))$. More precisely, let $ \K_{Fl,n}$ denote the 2-category with objects as before and with morphisms given by kernels in $D(T^*Fl(\uk) \times T^* Fl(\uk'))$. Then \cite{CK3} provides an $ (\gl_n,\t) $ action on $ \K_{Fl,n}$, which categorifies the $ U_q(\gl_n)$ representation $ (\C^n)^{\otimes n} $. This action can be extended to an $(L\gl_n,\t)$ action by using line bundles.

This leads us to the following statement.

\begin{Conjecture}
Linear Koszul duality (Theorem \ref{th:KoszulDuality}) gives us an equivalence $ \tK_{Fl,n}^{DG} \rightarrow \K_{Fl,n} $ which intertwines the two $(L\gl_n, \t)$ actions on either side.
\end{Conjecture}

In fact, we believe that this conjecture should not be too hard to deduce from the results of Riche \cite{Ri}. A forthcoming paper by Nandakumar and Zhao \cite{NZ} will partially address this conjecture.

\begin{Remark}
If we do not have $k_1 + \dots + k_n = m$ then it is not clear to us how to make an analog of Theorem \ref{th:KoszulDuality}.  For one thing, the variety $ \Y(\uk)_1 $ will not usually be a vector bundle.  Also, $\Y(\uk) $ is nonempty for any choice of $k_i \in \N$, whereas $Y(\uk)$ is empty if $k_i > m$ for some $i$. Thus, in general, we cannot expect Koszul duality to relate coherent sheaves on $ \Y(\uk)_1 $ with those on a subset of $ Y(\uk) $. 
\end{Remark}

\end{document}